\documentclass[11pt]{amsart}
\usepackage[margin=1.15in]{geometry}
\usepackage{amsmath,amssymb,amsthm,amsfonts}
\usepackage{mathrsfs}
\usepackage{enumitem}
\usepackage[expansion=false]{microtype}
\usepackage{xcolor}
\usepackage[colorlinks=true,linkcolor=blue!60!black,citecolor=blue!60!black,urlcolor=blue!60!black]{hyperref}

\setlist[enumerate]{leftmargin=*,itemsep=2pt,topsep=3pt}

\theoremstyle{plain}
\newtheorem{theorem}{Theorem}[section]
\newtheorem{lemma}[theorem]{Lemma}
\newtheorem{proposition}[theorem]{Proposition}
\newtheorem{corollary}[theorem]{Corollary}

\theoremstyle{definition}

\newtheorem{remark}[theorem]{Remark}

\numberwithin{equation}{section}

\newcommand{\C}{\mathbb{C}}
\newcommand{\N}{\mathbb{N}}
\newcommand{\Z}{\mathbb{Z}}
\newcommand{\Q}{\mathbb{Q}}
\newcommand{\PP}{\mathbb{P}}
\newcommand{\K}{\mathbb{K}}
\newcommand{\sD}{\mathcal{D}}
\newcommand{\sO}{\mathcal{O}}
\newcommand{\sM}{\mathcal{M}}
\newcommand{\sN}{\mathcal{N}}
\newcommand{\sV}{\mathcal{V}}
\newcommand{\sL}{\mathcal{L}}
\newcommand{\sI}{\mathcal{I}}
\newcommand{\Ann}{\operatorname{Ann}}
\newcommand{\rank}{\operatorname{rank}}
\newcommand{\Chrel}{\operatorname{Ch}^{\mathrm{rel}}}
\newcommand{\CCrel}{\operatorname{CC}^{\mathrm{rel}}}
\newcommand{\CC}{\operatorname{CC}}
\newcommand{\Ch}{\operatorname{Ch}}
\newcommand{\gr}{\operatorname{gr}}
\newcommand{\Hom}{\operatorname{Hom}}

\newcommand{\coker}{\operatorname{coker}}

\newcommand{\Spec}{\operatorname{Spec}}
\newcommand{\supp}{\operatorname{supp}}
\newcommand{\Der}{\operatorname{Der}}
\newcommand{\Frac}{\operatorname{Frac}}
\newcommand{\Hilb}{\operatorname{Hilb}}
\newcommand{\Span}{\operatorname{span}}
\newcommand{\codim}{\operatorname{codim}}
\newcommand{\im}{\operatorname{im}}
\newcommand{\ev}{\operatorname{ev}}

\newcommand{\Exp}{\operatorname{Exp}}
\newcommand{\one}{\mathbf{1}}
\newcommand{\mon}[1]{\langle #1\rangle}
\newcommand{\ol}[1]{\overline{#1}}
\newcommand{\wt}[1]{\widetilde{#1}}

\begin{document}

\title[Bernstein-Sato ideals for free hyperplane arrangements]{Bernstein-Sato ideals for free hyperplane arrangements}
\author{Wenzong Guo, Lei Wu, Fanghan Xiang}

\address{Wenzong Guo, School of Mathematical Sciences, Zhejiang
University, Hangzhou, China}
\email{wenzongguo@zju.edu.cn}
\address{Lei Wu, School of Mathematical Sciences, Zhejiang
University, Hangzhou, China}
\email{leiwu23@zju.edu.cn}
\address{Fanghan Xiang, School of Mathematical Sciences, Zhejiang
University, Hangzhou, China}
\email{fanghanxiang@zju.edu.cn}

\subjclass[2020]{14F10, 32S22, 32S40, 52C35}

\begin{abstract}
Let $f=(f_1,\dots,f_r)$ be a complete factorization of a central hyperplane arrangement $D$ in $X=\C^n$. For a monoid ideal $K\subseteq \N^r$ we study the Bernstein-Sato ideal $B^K_f$ of $f$ along $K$, that is, the $\C[s]$-annihilator of $\sD_X[s]f^s/\sum_{m\in K}\sD_X[s]f^{s+m}$. When $D$ is free we compute two families of these ideals with the help of AI. For the unit shift $K=\mon{e_i}$ we prove that $B^{-e_i}_f$ is generated by an explicit product of linear forms indexed by the dense edges of $D$ contained in $D_i$. This determines all the Bernstein-Sato ideals $B^{a,b}_f=\Ann_{\C[s]}\sD_X[s]f^{s-a}/\sD_X[s]f^{s-b}$, $a\geq b$, of a free arrangement, generalizing formulas of Maisonobe \cite{Mai16b} and Bath \cite{Bat20b}. The main new ingredient identifies the multiplicities of the relative characteristic cycle of $\sD_X[s]f^s/\sD_X[s]f^{s+e_i}$ along the conormal bundle of the origin with the coefficients of the Hilbert series of an Artinian complete intersection attached to a generic Ziegler restriction of $D$; the total multiplicity computed in \cite{Wu20} then forces all the resulting coefficientwise upper bounds to be equalities. For the coordinate monoid ideal $K=\mon{e_1,\dots,e_r}$ we show that $B^K_f$ is generated by one Euler relation for each irreducible factor of the essential quotient of $D$. Finally, we show that the zero locus of a Bernstein-Sato ideal along a monoid ideal need not be a finite union of translated linear subvarieties, even for a reduced free arrangement in $\C^2$: for $f=(x,y,x+y,x+2y)$ and $K=\mon{3e_1,3e_2}$ we compute $B^K_f$ exactly and find an irreducible quadric component. This disproves a conjecture due to Budur.
\end{abstract}

\maketitle

\setcounter{tocdepth}{2}
\tableofcontents

\section{Introduction}

\subsection{Bernstein-Sato ideals along monoid ideals}\label{sec:intro-monoid}
Let $X$ be a smooth complex algebraic variety (or a complex manifold) and let $f=(f_1,\dots,f_r)$ be an $r$-tuple of regular (resp. holomorphic) functions on $X$. We write $s=(s_1,\dots,s_r)$, $\C[s]=\C[s_1,\dots,s_r]$, $\sD_X[s]=\sD_X\otimes_\C\C[s]$ and $f^s=\prod_{i=1}^r f_i^{s_i}$. Let $U=X\setminus (\prod_i f_i=0)$ and $j\colon U\hookrightarrow X$. For $a\in\Z^r$ we denote by $\sD_X[s]f^{s+a}$ the $\sD_X[s]$-submodule of $j_*(\sO_U[s]f^s)$ generated by $f^{s+a}=\prod_i f_i^{s_i+a_i}$. Following \cite[\S 3.1]{Wu20}, for $a\geq b$ in $\Z^r$ (coordinatewise) we set
\[
M^{a,b}_f:=\frac{\sD_X[s]f^{s-a}}{\sD_X[s]f^{s-b}},\qquad B^{a,b}_f:=\Ann_{\C[s]}M^{a,b}_f ,
\]
and we abbreviate $M^{0,b}_f$, $B^{0,b}_f$ to $M^{b}_f$, $B^{b}_f$. The ideal $B_f:=B^{-\one_r}_f$, where $\one_r=(1,\dots,1)$, is the classical Bernstein-Sato ideal of $f$: it consists of the polynomials $b(s)$ satisfying
\[
b(s)f^s=P\cdot f^{s+\one_r}\qquad\text{for some } P\in \sD_X[s].
\]
When $r=1$ its monic generator is the $b$-function of $f$.

A \emph{monoid ideal} of $\N^r$ is a subset $K\subseteq \N^r$ with $K+\N^r\subseteq K$. By Dickson's lemma every monoid ideal is finitely generated, that is,
\[
K=\mon{m_1,\dots,m_p}:=\sum_{k=1}^p (m_k+\N^r)
\]
for finitely many $m_1,\dots,m_p\in\N^r$. We define the \emph{Bernstein-Sato ideal of $f$ along $K$} to be
\begin{equation}\label{eq:BK}
B^K_f:=\Ann_{\C[s]}N^K_f,\qquad\text{where}\qquad N^K_f:=\frac{\sD_X[s]f^s}{\sum_{m\in K}\sD_X[s]f^{s+m}} .
\end{equation}
Since $f^{s+m+u}=f^{u}f^{s+m}$ for $u\in\N^r$, only the generators of $K$ matter:
\begin{equation}\label{eq:generators-suffice}
\sum_{m\in K}\sD_X[s]f^{s+m}=\sum_{k=1}^p\sD_X[s]f^{s+m_k}.
\end{equation}
For instance, $N^{\mon{\one_r}}_f=M^{-\one_r}_f$ and $B^{\mon{\one_r}}_f=B_f$; for $i\in\{1,\dots,r\}$ and the $i$-th unit vector $e_i$,
\[
N^{\mon{e_i}}_f=M^{-e_i}_f=\frac{\sD_X[s]f^s}{\sD_X[s]f^{s+e_i}},\qquad B^{\mon{e_i}}_f=B^{-e_i}_f ,
\]
which we call the \emph{unit-shift} module and ideal; and for the \emph{coordinate monoid ideal} $K_{[r]}:=\mon{e_1,\dots,e_r}$ the module $N^{K_{[r]}}_f$ is the module $N_0$ of \cite[\S 5.3]{Wu20}. More generally, for a nonempty subset $I\subseteq\{1,\dots,r\}$ we write $K_I:=\mon{e_i\mid i\in I}$.

The ideals $B^K_f$ were introduced in \cite{Wu26} in the more general context of a regular holonomic coefficient module $\sM$ on $X$; the ideal $B^K_f$ is the ideal $B^{K}(\sN_0)$ of \emph{loc. cit.} for the coefficient module $\sM=\sO_X$, and its zero locus governs the generalized nearby cycles of $\sO_X$ along the monoid ideal $K$. Based on computational examples, Budur \cite{Bud15} conjectured that Bernstein-Sato ideals of this type are generated by products of linear polynomials $c_1s_1+\dots+c_rs_r+c_0$ with $c_i\in\N$ and $c_0>0$ (see also \cite[Remark.6.8.(1)]{BSZ25} and \cite[\S 1.3]{Wu26}); for the classical ideal $B_f$, the theorems of Sabbah \cite{Sab87}, Gyoja \cite{Gyo93} and Maisonobe \cite{Mai16a} together with \cite{BVWZ19} show that the codimension-one components of $Z(B_f)$ are hyperplanes of this form. 


In this paper we compute $B^K_f$ exactly for several monoid ideals $K$ when $f$ is a complete factorization of a free hyperplane arrangement with the help of AI. The answers for unit shifts and for coordinate monoid ideals are generated by linear forms, in accordance with the above expectations; the answer for $f=(x,y,x+y,x+2y)$ and $K=\mon{3e_1,3e_2}$ is not, and it shows that Budur's conjecture fails already for a reduced free line arrangement. Since $K=\mon{3e_1,3e_2}$ is not principal, it is still not known if the original conjecture \cite[Conjecture 1.1] {Bud15} holds. We suspect that AI would help in finding further counterexamples even for principal $K$. 

\subsection{Hyperplane arrangements}\label{sec:intro-arr}
A hyperplane arrangement $D=\{D_1,\dots,D_r\}$ is a finite collection of hyperplanes in $X=\C^n$; throughout this paper all arrangements are \emph{central} (every $D_j$ is a linear subspace) and \emph{reduced} (the $D_j$ are pairwise distinct). We choose linear forms $f_j$ with $D_j=(f_j=0)$ and put $f_D=\prod_j f_j$. The $r$-tuple $f=(f_1,\dots,f_r)$ is called a \emph{complete factorization} of $D$ (or of $f_D$); it is unique up to nonzero constants and up to reordering, and none of the statements below depends on these choices. The \emph{intersection lattice} $L(D)$ is the set of all intersections of members of $D$; its elements are called \emph{edges}. For an edge $W$ we write
\[
J(W,f)=\{j\in\{1,\dots,r\}\mid W\subseteq D_j\},\qquad D_W=\{D_j\mid j\in J(W,f)\},
\]
\[
D^W=\{D_j/W\mid j\in J(W,f)\},
\]
and $\rank(W)=\codim_X W$. Thus $D^W$ is a central essential arrangement in $X/W$, and $D_W$ is the pullback of $D^W$ under the projection $X\to X/W$. We write $f_W=(f_j)_{j\in J(W,f)}$ and regard it both as a complete factorization of $D_W$ on $X$ and of $D^W$ on $X/W$. The arrangement $D$ is \emph{essential} if $\{0\}\in L(D)$, and \emph{irreducible} if there is no linear change of coordinates on $\C^n$ such that $f_D$ becomes a product of two nonconstant polynomials in disjoint sets of variables. An edge $W$ is \emph{dense} if $D^W$ is irreducible. Finally, $D$ is \emph{free} if the underlying divisor is free in the sense of K. Saito \cite{Sai80}, that is, if the module $\Der(-\log D)$ of logarithmic vector fields is free over $\C[x_1,\dots,x_n]$; the degrees $d_1\leq\dots\leq d_n$ of a homogeneous basis are the \emph{exponents} of $D$ (see \S\ref{sec:free}).

By Maisonobe \cite{Mai16b} (see \cite[Theorem 5.3]{Wu20}), for a free arrangement $D$ with complete factorization $f$ the Bernstein-Sato ideal $B_f$ is principal and generated by
\begin{equation}\label{eq:Maisonobe}
\prod_{W\in L(D)\ \mathrm{dense}}\ \prod_{j=0}^{2(|J(W,f)|-\rank(W))}\Big(\sum_{i\in J(W,f)}s_i+\rank(W)+j\Big).
\end{equation}
Bath \cite{Bat23} proved that the Bernstein-Sato ideals $B^{-a}_f$ of tame (in particular of free) arrangements are principal and radical for every $a\in\N^r\setminus\{0\}$, and described their zero loci in several cases \cite{Bat20b,Bat23}. Our first result gives the generator of the unit-shift ideal $B^{-e_i}_f$ of a free arrangement exactly; together with Bath's theorems it determines all the ideals $B^{a,b}_f$ of a free arrangement.

\subsection{Unit shifts Bernstein-Sato ideals for free arrangements}\label{sec:intro-unit}
For a central arrangement $D$ with complete factorization $f$ and for $i\in\{1,\dots,r\}$ we define the squarefree polynomial
\begin{equation}\label{eq:PDi}
P_{D,i}(s):=\prod_{\substack{W\in L(D)\ \mathrm{dense}\\ i\in J(W,f)}}\ \prod_{\nu=0}^{|J(W,f)|-2\rank(W)+1}\Big(\sum_{j\in J(W,f)}s_j+\rank(W)+\nu\Big).
\end{equation}
For a dense edge $W$ of a free arrangement the upper limit $|J(W,f)|-2\rank(W)+1$ is nonnegative; see \eqref{eq:rho}.

\begin{theorem}\label{thm:main-unit}
Let $D$ be a free central hyperplane arrangement in $X=\C^n$ with complete factorization $f=(f_1,\dots,f_r)$. Then for every $i\in\{1,\dots,r\}$,
\[
B^{-e_i}_f=\big(P_{D,i}(s)\big).
\]
\end{theorem}

Theorem \ref{thm:main-unit} governs all the ideals $B^{a,b}_f$. For $c\in\N^r$ and an edge $W$ we write $c_W:=\sum_{j\in J(W,f)}c_j$, and we define
\begin{equation}\label{eq:PDc}
P_{D,c}(s):=\prod_{\substack{W\in L(D)\ \mathrm{dense}\\ c_W>0}}\ \prod_{k=0}^{|J(W,f)|-2\rank(W)+c_W}\Big(\sum_{j\in J(W,f)}s_j+\rank(W)+k\Big),
\end{equation}
so that $P_{D,e_i}=P_{D,i}$. By the theorems of Bath \cite[Theorems 1.1 and 1.4]{Bat23} recalled in Theorem \ref{thm:Bath-principal-radical}, the ideals $B^{a,b}_f$ of a free arrangement are principal and radical, and by \cite[Proposition 3.3]{Wu20} their zero loci are additive along chains $a\geq b\geq c$. Combining this with Theorem \ref{thm:main-unit} we obtain:

\begin{corollary}\label{cor:general-shift}
Let $D$ be a free central hyperplane arrangement in $X=\C^n$ with complete factorization $f$, let $a\geq b$ in $\Z^r$ with $a\neq b$, and put $c=a-b\in\N^r\setminus\{0\}$. Then $B^{a,b}_f$ is principal and radical, and
\[
B^{a,b}_f=\big(P_{D,c}(s-a)\big).
\]
In particular $B^{-c}_f=\Ann_{\C[s]}\big(\sD_X[s]f^s/\sD_X[s]f^{s+c}\big)=(P_{D,c}(s))$ for every $c\in\N^r\setminus\{0\}$, and $Z(B^{a,b}_f)$ is determined by the intersection lattice of $D$ and by $a$, $b$.
\end{corollary}

For $c=e_i$ Corollary \ref{cor:general-shift} is Theorem \ref{thm:main-unit}, and for $c=\one_r$ (so that $c_W=|J(W,f)|$) it is Maisonobe's formula \eqref{eq:Maisonobe}; in this sense the unit-shift ideals are the building blocks of all the Bernstein-Sato ideals $B^{a,b}_f$ of a free arrangement. Thus, it also generalizes Bath's formula for $B^g_{f'L}$ in \cite[Theorem 1.4]{Bat20b}. 

Theorem \ref{thm:main-unit} is deduced from a local statement about relative characteristic cycles (see \S\ref{sec:relD} for the notation). Let $D$ be an essential irreducible free arrangement with exponents $(1,d_2,\dots,d_n)$; exactly one exponent equals $1$ by Lemma \ref{lem:exponent-one}. Saito's criterion gives $r=1+\sum_{a=2}^n d_a$, so that
\begin{equation}\label{eq:rho}
\rho:=r-2n+1=\sum_{a=2}^n(d_a-2)\geq 0 .
\end{equation}
For $\nu\in\Z$ we put
\begin{equation}\label{eq:lambda-nu}
\lambda_\nu(s):=\sum_{j=1}^r s_j+n+\nu,\qquad \Gamma_\nu:=T^*_{\{0\}}X\times(\lambda_\nu=0)\subseteq T^*X\times\C^r,
\end{equation}
and we denote by $m_{\Gamma_\nu}(M^{-e_i}_f)$ the multiplicity of $\Gamma_\nu$ in the relative characteristic cycle of $M^{-e_i}_f$, with the convention that the multiplicity is $0$ if $\Gamma_\nu$ is not a component of the relative characteristic variety of $M^{-e_i}_f$.

\begin{theorem}\label{thm:main-local}
Let $D$ be an essential, irreducible, free central hyperplane arrangement in $X=\C^n$ with exponents $(1,d_2,\dots,d_n)$ and complete factorization $f=(f_1,\dots,f_r)$, and let $i\in\{1,\dots,r\}$. Then for every $\nu\in\Z$,
\begin{equation}\label{eq:main-local}
m_{\Gamma_\nu}(M^{-e_i}_f)=[t^\nu]\prod_{a=2}^n\big(1+t+\dots+t^{d_a-2}\big),
\end{equation}
where $[t^\nu]h(t)$ denotes the coefficient of $t^\nu$ in $h(t)$. Moreover, for $\alpha\in\C$ the hyperplane $(\sum_{j=1}^r s_j+\alpha=0)$ is contained in $Z(B^{-e_i}_f)$ if and only if $\alpha=n+\nu$ with $0\leq\nu\leq\rho$.
\end{theorem}

The right-hand side of \eqref{eq:main-local} is the Hilbert function of an Artinian complete intersection of $\C[x_2,\dots,x_n]$ cut out by a regular sequence of degrees $d_2-1,\dots,d_n-1$; it is positive exactly for $0\leq\nu\leq\rho$, and its total sum is $\prod_{a=2}^n(d_a-1)=(-1)^{n-1}\chi(\PP(X)\setminus\bigcup_{H\in D}\PP(H))$, which is the multiplicity computed in \cite[Theorem 1.5]{Wu20}. The proof of Theorem \ref{thm:main-local} goes as follows. Using a Saito basis adapted to $D_i$ and Ziegler's restriction theorem, we bound $m_{\Gamma_\nu}(M^{-e_i}_f)$ from above by the coefficient $h_\nu$ of the Hilbert series of a complete intersection $\C[x_2,\dots,x_n]/(g_2,\dots,g_n)$, where the $g_a$ are generic linear combinations of the logarithmic quotients $\theta_a(f_j)/f_j$; the point is that after localizing $\C[s]$ at the prime $(\lambda_\nu)$ the module $M^{-e_i}_f$ becomes a point-supported module over a field, and relative Kashiwara's equivalence turns the multiplicity into the dimension of a space of solutions of a system of linear equations depending on $s$, whose leading part in a suitable direction is the transpose of the multiplication map by the $g_a$. On the other hand, the multiplicity of $T^*_{\{0\}}X\times(\sum_j s_j+n=0)$ in $\CCrel(\sD_X[s]f^{s-k_r}/\sD_X[s]f^{s+k_r})$, $k\gg 0$, is the sum $\sum_\nu m_{\Gamma_\nu}(M^{-e_i}_f)$ by the additivity of relative characteristic cycles, and it equals $\sum_\nu h_\nu$ by \cite[Theorem 1.5]{Wu20} and Terao's factorization theorem. Hence all the upper bounds are equalities.

Theorem \ref{thm:main-local} refines \cite[Theorem 5.6]{Wu20}, which says that $(\sum_j s_j+n=0)$ is a component of $Z(B^{-e_i}_f)$. Together with the localization at the edges and a lemma on translation-invariant components (Lemma \ref{lem:ACC}), it gives both the lower and the upper bound in Theorem \ref{thm:main-unit}. 

\subsection{Coordinate monoid ideals}\label{sec:intro-coord}
Let $D$ be a central arrangement with complete factorization $f$, let $W_0=\bigcap_{j=1}^r D_j$ be its center and let $\ell=\rank(W_0)$. The essential arrangement $D^{W_0}$ in $X/W_0$ decomposes uniquely, after a linear change of coordinates, as a product
\begin{equation}\label{eq:product-decomp}
D^{W_0}=D^{(1)}\times\dots\times D^{(c)},\qquad X/W_0=V_1\oplus\dots\oplus V_c,
\end{equation}
of essential irreducible central arrangements $D^{(a)}$ in $V_a$. Let $I_a\subseteq\{1,\dots,r\}$ be the set of indices $j$ such that $D_j/W_0$ belongs to $D^{(a)}$ and let $n_a=\dim V_a$, so that $\{1,\dots,r\}=I_1\sqcup\dots\sqcup I_c$ and $\ell=n_1+\dots+n_c$. We define the \emph{Euler relations}
\begin{equation}\label{eq:Euler-relations}
\lambda^{(a)}:=\sum_{i\in I_a}s_i+n_a\qquad (a=1,\dots,c).
\end{equation}

\begin{theorem}\label{thm:main-coord}
Let $D$ be a free central hyperplane arrangement in $X=\C^n$ with complete factorization $f$. Then, with notation as above,
\[
B^{K_{[r]}}_f=\Ann_{\C[s]}N^{K_{[r]}}_f=\big(\lambda^{(1)},\dots,\lambda^{(c)}\big).
\]
In particular, if $D$ is essential and irreducible, then $B^{K_{[r]}}_f=(\sum_{i=1}^r s_i+n)$.
\end{theorem}

The last assertion sharpens \cite[Lemma 5.5]{Wu20}, which says that $\sum_i s_i+n\in B^{K_{[r]}}_f$ whenever a subset of the $f_j$ is a coordinate system. The proof of the above theorem also depends on relative Kashiwara's equivalence as the $\mathcal O_X$-module support of $N^{K_{[r]}}_f$ is the center of $D$.



\subsection{A nonlinear component}\label{sec:intro-counter}
Our last result concerns the reduced free arrangement of four lines
\[
f=(f_1,f_2,f_3,f_4)=(x,y,x+y,x+2y)\quad\text{in } X=\C^2,
\]
and the monoid ideal $K=\mon{3e_1,3e_2}$, so that $N^K_f=\sD_X[s]f^s/(\sD_X[s]x^3f^s+\sD_X[s]y^3f^s)$. For $d\in\Z$ let $\lambda_d=s_1+s_2+s_3+s_4+2+d$ as in \eqref{eq:lambda-nu} (here $n=2$), and put
\begin{equation}\label{eq:Phi-intro}
\Phi:=(2s_1-s_2+s_3+1)(2s_1-s_2+s_3+4)+2(s_1+2)(s_2+2).
\end{equation}

\begin{theorem}\label{thm:main-counter}
With notation as above,
\[
B^K_f=(\lambda_0)\cap(\lambda_1)\cap(\lambda_2)\cap(\lambda_3,\Phi)\cap(s_1+3,\,s_2+3,\,s_3,\,s_4).
\]
The ideal $B^K_f$ is radical, and $(\lambda_3,\Phi)$ is a prime ideal defining an irreducible nonlinear quadric surface in the hyperplane $(\lambda_3=0)$. 
\end{theorem}

For comparison, Theorems \ref{thm:main-unit} and \ref{thm:main-coord} give $B^{-e_1}_f=\big((s_1+1)\lambda_0\lambda_1\big)$ and $B^{K_{[4]}}_f=(\lambda_0)$ for the same arrangement. 
Note that the exponential image $\Exp(Z(B^K_f))\subseteq(\C^*)^4$ of the zero locus in Theorem \ref{thm:main-counter} is nevertheless the subtorus $\{t_1t_2t_3t_4=1\}$, in accordance with \cite[Theorem A]{Wu26}; see Remark \ref{rem:conjectures}. The proof of Theorem \ref{thm:main-counter} also relies on the relative Kashiwara's equivalence: the relations $x^3=y^3=0$ given by $K$ reduce the computation to linear algebra on a nine-dimensional truncation of the $\delta$-module at the origin, and a Yoneda-type argument (Proposition \ref{prop:universal-jet}) shows that the resulting matrices present the actual finite $\C[s]$-module underlying $N^K_f$, not merely its fibers. Consequently the ideal itself is obtained, not only its radical. Theorem \ref{thm:main-counter} also indicate that one could not expect a general formula of $B^K_f$ even for free hyperplane arrangements. 

\subsection{Organization}
In Section \ref{sec:prelim} we fix notation and collect the results on relative $\sD$-modules, on free arrangements, on logarithmic annihilators and on Bernstein-Sato ideals that are used later.
Section \ref{sec:local} proves Theorem \ref{thm:main-local}, and Section \ref{sec:global} deduces Theorem \ref{thm:main-unit} and Corollary \ref{cor:general-shift}. Theorems \ref{thm:main-coord} is proved in Section \ref{sec:coord}. Section \ref{sec:counter} gives Theorem \ref{thm:main-counter}.

\subsubsection*{Statement on the use of AI systems}
Two of the results of this paper were obtained with the help of automated reasoning systems, and we record here what each of them contributed. The proof of Theorem \ref{thm:main-local}, that is, the local multiplicity theorem of Section \ref{sec:local} together with the auxiliary statements of \S\S\ref{sec:adapted-basis}--\ref{sec:total} on which it rests, was produced by the Rethlas system \cite{Rethlas} running GPT-5.6 Sol. The example underlying Theorem \ref{thm:main-counter} was found by Danus \cite{Danus} running GPT-5.6 Sol and Fable 5; the computation of $B^K_f$ presented in Section \ref{sec:counter} has the same origin. All statements and all proofs have been checked and rephrased by the authors, who are solely responsible for the correctness of the results and for the final form of the text.


\section{Preliminaries}\label{sec:prelim}

\subsection{Relative $\sD$-modules}\label{sec:relD}
We briefly recall the notation of \cite[\S 2]{Wu20}; see also \cite{BVWZ19,Mai16a}. Let $X$ be a smooth complex algebraic variety (or a complex manifold) of dimension $n$ and let $R$ be a commutative noetherian $\C$-algebra which is a domain of finite Krull dimension and whose localizations at prime ideals are regular; in this paper $R$ is always $\C[s]$, a localization of $\C[s]$, or a field extension of $\C$. Put $A_R=\sD_X\otimes_\C R$ with $\sD_X$ the sheaf of (algebraic or analytic) differential operators on $X$; thus $A_{\C[s]}=\sD_X[s]$. A coherent $A_R$-module $M$ admits (locally) a relative good filtration $F_\bullet^{\mathrm{rel}}M$, compatible with the filtration $F_\bullet\sD_X\otimes_\C R$ of $A_R$ by the order of differential operators (elements of $R$ have degree zero), such that $\gr^{\mathrm{rel}}_\bullet M$ is coherent over $\gr_\bullet\sD_X\otimes_\C R\simeq\sO_{T^*X}\otimes_\C R$. The support of $\gr^{\mathrm{rel}}_\bullet M$ in $T^*X\times\Spec R$ is the \emph{relative characteristic variety} $\Chrel(M)$, and the cycle
\[
\CCrel(M)=\sum_{\Gamma}m_\Gamma(M)\,[\Gamma],
\]
where $\Gamma$ runs over the irreducible components of $\Chrel(M)$ and $m_\Gamma(M)$ is the length of $\gr^{\mathrm{rel}}_\bullet M$ at the generic point of $\Gamma$, is the \emph{relative characteristic cycle}; both are independent of the good filtration. We write $\CCrel_k(M)$ for the part of $\CCrel(M)$ of pure dimension $k$, and we put $m_\Gamma(M)=0$ if $\Gamma$ is not a component of $\Chrel(M)$. When $R$ is a field we simply write $\Ch$ and $\CC$. The module $M$ is \emph{relative holonomic} if every irreducible component of $\Chrel(M)$ is of the form $\Lambda\times S$ with $\Lambda\subseteq T^*X$ an irreducible conic Lagrangian subvariety and $S\subseteq\Spec R$ irreducible \cite[Definition 3.2.3]{BVWZ19}. Relative holonomic modules form an abelian subcategory of the category of coherent $A_R$-modules \cite[3.2.4]{BVWZ19}; in particular quotients of relative holonomic modules are relative holonomic.

The \emph{Bernstein-Sato ideal} of a coherent $A_R$-module $M$ is $B_M:=\Ann_R(M)$. For a prime ideal $\mathfrak q\subseteq R$ we write $M_{\mathfrak q}=M\otimes_R R_{\mathfrak q}$, a coherent $A_{R_{\mathfrak q}}$-module. Since $R$ is central, $B_{M_{\mathfrak q}}=(B_M)_{\mathfrak q}$, and since localization is exact, a relative good filtration of $M$ localizes to one of $M_{\mathfrak q}$, so that
\begin{equation}\label{eq:localize-CC}
\CCrel(M_{\mathfrak q})=\CCrel(M)_{\mathfrak q},
\end{equation}
where the right-hand side is the sum of the terms $m_\Gamma(M)[\Gamma]$ over the components $\Gamma$ whose projection to $\Spec R$ contains $\mathfrak q$; see \cite[\S 2.2]{Wu20}. We write $\supp_R(M)$ for the set of primes $\mathfrak q$ with $M_{\mathfrak q}\neq 0$ and $p_2\colon T^*X\times\Spec R\to\Spec R$ for the projection.

\begin{lemma}[{\cite[Lemma 3.4.1]{BVWZ19}, \cite[Lemma 2.2]{Wu20}}]\label{lem:Z=supp}
If $M$ is a relative holonomic $A_R$-module, then $Z(B_M)=p_2(\Chrel(M))$ and $Z(B_M)=\supp_R(M)$.
\end{lemma}

Now let $f=(f_1,\dots,f_r)$ be an $r$-tuple of regular functions on $X$ and let $j\colon U=X\setminus(\prod_i f_i=0)\hookrightarrow X$. For $a\in\Z^r$ the map
\begin{equation}\label{eq:translation}
\sigma_a\colon j_*(\sO_U[s]f^s)\to j_*(\sO_U[s]f^s),\qquad g(s)f^s\mapsto g(s-a)f^{s-a},
\end{equation}
is a $\sD_X$-linear bijection satisfying $\sigma_a(b(s)m)=b(s-a)\sigma_a(m)$, and it maps $\sD_X[s]f^{s-b}$ onto $\sD_X[s]f^{s-a-b}$. Hence $\sigma_a$ induces an isomorphism $M^{b,c}_f\simeq M^{a+b,a+c}_f$ of $\sD_X$-modules which is semilinear with respect to the automorphism $\tau_a\colon b(s)\mapsto b(s-a)$ of $\C[s]$. Consequently $B^{a+b,a+c}_f=\tau_a(B^{b,c}_f)$, $Z(B^{a+b,a+c}_f)=Z(B^{b,c}_f)+a$, and $\CCrel(M^{a+b,a+c}_f)$ is the image of $\CCrel(M^{b,c}_f)$ under the translation by $a$ of the second factor of $T^*X\times\C^r$. We refer to these facts as ``translation'' or ``substitution''. The same applies to the modules $N^K_f$.

\begin{theorem}[{\cite[Theorem 3.2 and Proposition 3.3]{Wu20}}]\label{thm:Wu-basic}
Let $a\geq b\geq c$ in $\Z^r$ with $a\neq b$.
\begin{enumerate}
\item $\sD_X[s]f^{s+a}$ is relative holonomic and $\CCrel(\sD_X[s]f^{s+a})=\CC(j_*\sO_U)\times\C^r$.
\item $M^{a,b}_f$ is relative holonomic, $\dim\Chrel(M^{a,b}_f)=n+r-1$ (if $M^{a,b}_f\not=0$), and $p_2(\Chrel(M^{a,b}_f))=Z(B^{a,b}_f)$.
\item $\CCrel_{n+r-1}(M^{a,c}_f)=\CCrel_{n+r-1}(M^{a,b}_f)+\CCrel_{n+r-1}(M^{b,c}_f)$, and consequently
\[
Z_{r-1}(B^{a,c}_f)=Z_{r-1}(B^{a,b}_f)\cup Z_{r-1}(B^{b,c}_f),
\]
where $Z_{r-1}$ denotes the part of pure dimension $r-1$ of the zero locus.
\end{enumerate}
\end{theorem}

Since $j_*\sO_U$ is a regular holonomic $\sD_X$-module whose de Rham complex is constructible with respect to any Whitney stratification of $X$ adapted to $\bigcup_i(f_i=0)$, part (1) shows that, when $f$ is a complete factorization of a hyperplane arrangement $D$, the relative characteristic varieties of all the modules $\sD_X[s]f^{s+a}$ and $M^{a,b}_f$ are contained in $\bigcup_{W\in L(D)}T^*_WX\times\C^r$, where $T^*_WX$ denotes the conormal bundle of the edge $W$.

\begin{theorem}[{\cite[Theorem 1.5]{Wu20}}]\label{thm:Wu-dense}
Let $f$ be a complete factorization of a central hyperplane arrangement $D$ and let $W$ be a dense edge of $D$. For $l\in\Z$ and $k\gg l$ put $k_r=(k,\dots,k)\in\Z^r$. Then
\[
T^*_WX\times\Big(\sum_{j\in J(W,f)}s_j+l=0\Big)
\]
is a component of $\Chrel(M^{k_r,-k_r}_f)$, and its multiplicity in $\CCrel(M^{k_r,-k_r}_f)$ is
\[
(-1)^{\rank(W)-1}\chi\Big(\PP(X/W)\setminus\bigcup_{H\in D^W}\PP(H)\Big)>0 ,
\]
where $\chi$ denotes the topological Euler characteristic.
\end{theorem}

\begin{theorem}[{\cite[Theorem 3.12]{Wu20}}]\label{thm:Wu-jstar}
For every prime ideal $\mathfrak q\subseteq\C[s]$ one has $j_*(\sO_U[s]f^s)_{\mathfrak q}=\sD_X[s]_{\mathfrak q}f^{s-k_r}$ for all $k\gg 0$. In particular $j_*(\sO_U[s]f^s)_{\mathfrak q}$ is a coherent $\sD_X[s]_{\mathfrak q}$-module.
\end{theorem}

\begin{theorem}[{Sabbah \cite{Sab87}; see \cite[Theorem 3.4]{Wu20}}]\label{thm:Sabbah}
For every $r$-tuple $f$ (of regular functions, or of germs of holomorphic functions at a point), the ideal $B_f$ contains a nonzero polynomial which is a product of linear forms $L\cdot s+\alpha$ with $L\in\N^r$ and $\alpha\in\Q$.
\end{theorem}

\subsection{Relative Kashiwara's equivalence}\label{sec:kashiwara}
Let $X=\C^n$ with coordinates $x_1,\dots,x_n$ and let $\sD_X$ be the Weyl algebra. We write
\[
\Delta:=\sD_X/\sD_X(x_1,\dots,x_n)
\]
for the $\delta$-module at the origin, with cyclic vector $\delta_0$ (the class of $1$); the elements $\partial^\beta\delta_0=\partial_{x_1}^{\beta_1}\cdots\partial_{x_n}^{\beta_n}\delta_0$, $\beta\in\N^n$, form a $\C$-basis of $\Delta$, and
\begin{equation}\label{eq:delta-rules}
x_\alpha\,\partial^\beta\delta_0=-\beta_\alpha\,\partial^{\beta-e_\alpha}\delta_0,\qquad \partial_{x_\alpha}\,\partial^\beta\delta_0=\partial^{\beta+e_\alpha}\delta_0 .
\end{equation}
For a commutative $\C$-algebra $R$ we put $\Delta_R=\Delta\otimes_\C R$; when $R=\K$ is a field we also write $\delta_{0,\K}$ for $\Delta_\K$. The following elementary form of Kashiwara's equivalence \cite[Theorem 1.6.1]{HTT08} with coefficients will be used in Sections \ref{sec:local} and \ref{sec:counter}.

\begin{lemma}\label{lem:kashiwara}
Let $R$ be a commutative $\C$-algebra and let $N$ be a left $\sD_X\otimes_\C R$-module such that every element of $N$ is annihilated by a power of the ideal $(x_1,\dots,x_n)$. Let $V=\{v\in N\mid x_\alpha v=0\text{ for }\alpha=1,\dots,n\}$, an $R$-submodule of $N$.
\begin{enumerate}
\item The map $\Delta\otimes_\C V\to N$, $P\delta_0\otimes v\mapsto Pv$, is an isomorphism of $\sD_X\otimes_\C R$-modules. In particular $\Ann_R N=\Ann_R V$.
\item For every $R$-module $V'$, every $\sD_X\otimes_\C R$-linear map $\Psi\colon\Delta\otimes_\C V\to\Delta\otimes_\C V'$ is of the form $\mathrm{id}_\Delta\otimes\phi$ for a unique $R$-linear map $\phi\colon V\to V'$.
\item If $R=\K$ is a field and $N$ is finitely generated over $\sD_X\otimes_\C\K$, then $\dim_\K V<\infty$, $\CC(N)=(\dim_\K V)\,[T^*_{\{0\}}X]$, and $\dim_\K V=\dim_\K\Hom_{\sD_X\otimes\K}(N,\delta_{0,\K})$.
\end{enumerate}
\end{lemma}

\begin{proof}
(1) Let $A$ be any ring and let $N$ be a module over $\sD_{\C}\otimes_\C A$, where $\sD_\C=\C\langle x,\partial\rangle$ is the first Weyl algebra and $A$ commutes with $x$ and $\partial$, such that $x$ acts locally nilpotently on $N$. We claim that $\C[\partial]\otimes_\C\ker(x)\to N$, $\partial^k\otimes v\mapsto\partial^kv$, is bijective. Since $x\partial^kv=-k\partial^{k-1}v$ for $v\in\ker(x)$, an identity $\sum_{k\leq m}\partial^kv_k=0$ with $v_m\neq 0$ gives, after applying $x^m$, $(-1)^m m!\,v_m=0$, which is absurd; so the map is injective. For surjectivity we show by induction on $m$ that every $w$ with $x^{m+1}w=0$ lies in the image. If $m=0$ then $w\in\ker(x)$. Otherwise $xw$ is killed by $x^m$, so $xw=\sum_{k<m}\partial^kv_k$ with $v_k\in\ker(x)$, and $w'=w+\sum_{k<m}\frac{1}{k+1}\partial^{k+1}v_k$ satisfies $xw'=xw-\sum_{k<m}\partial^kv_k=0$; hence $w\in\ker(x)+\im$. Applying the claim to $x_1$ (with $A=\sD_{\C^{n-1}}\otimes R$) we get $N\simeq\C[\partial_{x_1}]\otimes_\C\ker(x_1)$, and $\ker(x_1)$ is a $\sD_{\C^{n-1}}\otimes R$-module on which $x_2,\dots,x_n$ act locally nilpotently; by induction $\ker(x_1)\simeq\C[\partial_{x_2},\dots,\partial_{x_n}]\otimes_\C V$, which proves the first assertion. As an $R$-module, $\Delta\otimes_\C V$ is a direct sum of copies of $V$, so $\Ann_RN=\Ann_RV$.

(2) By \eqref{eq:delta-rules} and (1), the common kernel of $x_1,\dots,x_n$ on $\Delta\otimes_\C V'$ is $\delta_0\otimes V'$. For $v\in V$ the element $\delta_0\otimes v$ is annihilated by all $x_\alpha$, hence so is $\Psi(\delta_0\otimes v)$, and therefore $\Psi(\delta_0\otimes v)=\delta_0\otimes\phi(v)$ for a unique $\phi(v)\in V'$. The map $\phi$ is $R$-linear because $\Psi$ is, and $\Psi=\mathrm{id}_\Delta\otimes\phi$ because $\Delta\otimes_\C V$ is generated by $\delta_0\otimes V$ over $\sD_X$.

(3) If $N$ is generated by finitely many elements, these lie in $\Delta\otimes_\C V_0$ for a finite-dimensional subspace $V_0\subseteq V$, so $N=\Delta\otimes_\C V_0$ and $V=V_0$ by (1). Hence $N\simeq\delta_{0,\K}^{\oplus\dim V}$. The order filtration of $\delta_{0,\K}$ has $\gr\,\delta_{0,\K}=\K[\xi_1,\dots,\xi_n]$, which is supported on $T^*_{\{0\}}X$ with multiplicity one; thus $\CC(N)=(\dim V)[T^*_{\{0\}}X]$. The last assertion follows from (2).
\end{proof}

\subsection{Free arrangements}\label{sec:free}
Let $D$ be a central arrangement in $X=\C^n$ with complete factorization $f$ and let $S_X=\C[x_1,\dots,x_n]$. The module of logarithmic vector fields is
\[
\Der(-\log D)=\{\theta\in\Der_\C(S_X)\mid \theta(f_j)\in f_jS_X\ \text{for } j=1,\dots,r\};
\]
it is a graded $S_X$-module, where the degree of $\theta=\sum_\alpha p_\alpha\partial_{x_\alpha}$ with homogeneous $p_\alpha$ of the same degree $d$ is $d$. The arrangement $D$ is \emph{free} if $\Der(-\log D)$ is a free $S_X$-module (necessarily of rank $n$), and the degrees of a homogeneous basis $\theta_1,\dots,\theta_n$ form the multiset of \emph{exponents} $\exp(D)$. The Euler vector field $E=\sum_{\alpha=1}^n x_\alpha\partial_{x_\alpha}$ satisfies $E(f_j)=f_j$ for every $j$. We shall use the following standard facts.

\begin{theorem}\label{thm:free-facts}
Let $D$ be a central arrangement in $X=\C^n$ with complete factorization $f$.
\begin{enumerate}
\item (Saito's criterion \cite{Sai80}, \cite[Theorem 4.19]{OT92}) Homogeneous $\theta_1,\dots,\theta_n\in\Der(-\log D)$ form a basis of $\Der(-\log D)$ if and only if $\det(\theta_1,\dots,\theta_n)=c\,f_D$ for some $c\in\C^*$, where $\det(\theta_1,\dots,\theta_n)$ is the determinant of the coefficient matrix. In that case $\sum_a\deg\theta_a=r$.
\item (\cite[Proposition 4.28]{OT92}) If $D=D_1\times D_2$ is a product of arrangements in complementary subspaces, then $D$ is free if and only if $D_1$ and $D_2$ are free, and then $\exp(D)=\exp(D_1)\cup\exp(D_2)$, a basis being given by the union of bases of $\Der(-\log D_1)$ and $\Der(-\log D_2)$.
\item (\cite[Theorem 4.37]{OT92}) If $D$ is free, then for every edge $W\in L(D)$ the localization $D_W$, and hence the essential quotient $D^W$, is free.
\item (Terao's factorization theorem \cite[Theorem 4.137]{OT92}) Suppose that $D$ is free with exponents $(d_1,\dots,d_n)$, and let $M(D)=X\setminus\bigcup_j D_j$. Then the Poincar\'e polynomial of the complement is
\[
\pi(M(D),t)=\sum_{q\geq 0}\dim H^q(M(D),\C)\,t^q=\prod_{a=1}^n(1+d_at).
\]
\end{enumerate}
\end{theorem}

\begin{lemma}\label{lem:exponent-one}
Let $D$ be an essential central arrangement in $X=\C^n$.
\begin{enumerate}
\item If $D$ is irreducible, then every logarithmic vector field with linear coefficients is a constant multiple of $E$, and no nonzero logarithmic vector field has constant coefficients.
\item If $D$ is free and irreducible, then exactly one exponent of $D$ equals $1$, and $D$ admits a homogeneous basis of the form $E,\theta_2,\dots,\theta_n$ with $\deg\theta_a\geq 2$ for $a\geq 2$.
\item If $D$ is free and $D=D^{(1)}\times\dots\times D^{(c)}$ is its decomposition into irreducible factors, with $D^{(a)}$ an arrangement in $V_a$, $X=V_1\oplus\dots\oplus V_c$, then $\Der(-\log D)$ admits a homogeneous basis of the form $\{E_a,\theta_{a,2},\dots,\theta_{a,n_a}\}_{a=1,\dots,c}$, where $n_a=\dim V_a$, $E_a$ is the Euler vector field of $V_a$ and $\deg\theta_{a,j}\geq 2$.
\end{enumerate}
\end{lemma}

\begin{proof}
(1) A vector field $\theta$ with linear coefficients maps linear forms to linear forms and thus defines a linear endomorphism $\phi_\theta$ of the space $X^*$ of linear forms. For degree reasons, $\theta$ is logarithmic if and only if $\phi_\theta(f_j)=c_jf_j$ for scalars $c_j$. Since $D$ is essential, the $f_j$ span $X^*$; hence $\phi_\theta$ is diagonalizable, $X^*=\bigoplus_\lambda X^*_\lambda$ is the direct sum of its eigenspaces, every $f_j$ lies in one of them, and every $X^*_\lambda$ is spanned by the $f_j$ it contains. If two distinct eigenvalues occurred, then in coordinates adapted to the decomposition $X^*=X^*_\lambda\oplus\bigoplus_{\lambda'\neq\lambda}X^*_{\lambda'}$ the polynomial $f_D$ would be a product of two nonconstant polynomials in disjoint sets of variables, contradicting the irreducibility of $D$. Hence $\phi_\theta=c\cdot\mathrm{id}$, that is, $\theta(h)=ch=cE(h)$ for all linear $h$, so $\theta=cE$. A constant vector field $\partial_v$ satisfies $\partial_v(f_j)=f_j(v)\in\C$, which lies in $(f_j)$ only if $f_j(v)=0$; since $D$ is essential this forces $v=0$.

(2) By (1) the homogeneous components of degree $0$ and $1$ of $\Der(-\log D)$ are $0$ and $\C E$. If $\theta_1,\dots,\theta_n$ is a homogeneous basis with degrees $d_a$, then no $d_a$ equals $0$, and the degree-one component has dimension $\#\{a\mid d_a=1\}$. Hence exactly one $d_a$ equals $1$, and the corresponding $\theta_a$ is a nonzero constant multiple of $E$.

(3) By Theorem \ref{thm:free-facts}(2) each $D^{(a)}$ is free and a basis of $\Der(-\log D)$ is obtained by taking the union of homogeneous bases of the $\Der(-\log D^{(a)})$; each $D^{(a)}$ is essential and irreducible in $V_a$, so (2) applies to it.
\end{proof}

\subsubsection*{Ziegler restriction}
Fix $H=D_i\in D$. On $H$ the hyperplanes $D_j\cap H$, $j\neq i$, need not be distinct; let $D|_H$ be the set of distinct hyperplanes of $H$ of this form. For $H'\in D|_H$ put $m_H(H')=\#\{j\neq i\mid D_j\cap H=H'\}$; the pair $(D|_H,m_H)$ is the \emph{Ziegler multiarrangement} \cite{Zie89}. If $\beta_{H'}$ is a linear form on $H$ defining $H'$, its module of logarithmic derivations is
\[
\Der(-\log(D|_H,m_H))=\{\eta\in\Der_\C(\C[H])\mid \eta(\beta_{H'})\in\beta_{H'}^{m_H(H')}\C[H]\ \text{for all } H'\in D|_H\},
\]
and $(D|_H,m_H)$ is free if this module is free; the degrees of a homogeneous basis are its exponents.

\begin{theorem}[{\cite[Theorems 8 and 11]{Zie89}}]\label{thm:Ziegler}
\begin{enumerate}
\item Homogeneous elements $\eta_2,\dots,\eta_n$ of $\Der(-\log(D|_H,m_H))$ form a basis if and only if
\[
\det(\eta_2,\dots,\eta_n)=c\prod_{H'\in D|_H}\beta_{H'}^{m_H(H')}\qquad\text{for some }c\in\C^*.
\]
\item If $D$ is free with exponents $(1,d_2,\dots,d_n)$ and $E,\theta_2,\dots,\theta_n$ is a homogeneous basis of $\Der(-\log D)$ with $\theta_a(f_i)=0$ for $a\geq 2$, then the restrictions $\ol\theta_2,\dots,\ol\theta_n$ of $\theta_2,\dots,\theta_n$ to $H$ form a homogeneous basis of $\Der(-\log(D|_H,m_H))$, of degrees $d_2,\dots,d_n$.
\end{enumerate}
\end{theorem}

\subsection{Logarithmic annihilators}\label{sec:log-ann}
For a basis $\delta_1,\dots,\delta_n$ of $\Der(-\log D)$ of a free arrangement $D$ with complete factorization $f$, the quotients $\delta_k(f_j)/f_j$ are polynomials, and we put
\begin{equation}\label{eq:tilde-delta}
\wt\delta_k:=\delta_k-\sum_{j=1}^r s_j\,\frac{\delta_k(f_j)}{f_j}\in\sD_X[s].
\end{equation}
Each $\wt\delta_k$ annihilates $f^s$. 

\begin{theorem}[{\cite[Proposition 2(3)]{Mai16b}}]\label{thm:log-ann}
Let $D$ be a free central arrangement with complete factorization $f$ and let $\delta_1,\dots,\delta_n$ be a basis of $\Der(-\log D)$. Then
\[
\Ann_{\sD_X[s]}f^s=\sD_X[s](\wt\delta_1,\dots,\wt\delta_n),
\]
and the same holds for the analytic stalks $\Ann_{\sD^{an}_{X,x}[s]}f^s=\sD^{an}_{X,x}[s](\wt\delta_1,\dots,\wt\delta_n)$ at every point $x\in X$.
\end{theorem}

\begin{lemma}\label{lem:cyclic}
Let $f$ be an $r$-tuple of regular functions on a smooth variety $X$ and let $K=\mon{m_1,\dots,m_p}$ be a monoid ideal. Then
\[
N^K_f\simeq\frac{\sD_X[s]}{\Ann_{\sD_X[s]}f^s+\sum_{k=1}^p\sD_X[s]f^{m_k}}\qquad\text{and}\qquad
B^K_f=\Big(\Ann_{\sD_X[s]}f^s+\sum_{k=1}^p\sD_X[s]f^{m_k}\Big)\cap\C[s],
\]
where $f^{m_k}=\prod_i f_i^{m_{k,i}}$. The same holds for the analytic stalks at every point.
\end{lemma}

\begin{proof}
The map $\sD_X[s]\to\sD_X[s]f^s$, $P\mapsto Pf^s$, has kernel $\Ann_{\sD_X[s]}f^s$. Since $f^{s+m_k}=f^{m_k}f^s$, the inverse image of the submodule $\sum_k\sD_X[s]f^{s+m_k}$, which equals $\sum_{m\in K}\sD_X[s]f^{s+m}$ by \eqref{eq:generators-suffice}, is the left ideal $\Ann_{\sD_X[s]}f^s+\sum_k\sD_X[s]f^{m_k}$. This gives the presentation. As $\C[s]$ is central in $\sD_X[s]$, a polynomial $b(s)$ annihilates the cyclic module if and only if it annihilates the class of $1$, that is, if and only if it lies in the displayed left ideal.
\end{proof}

We now fix $X=\C^n$. Following \cite[Remark 3.1]{Wu20}, we use analytic stalks for local statements: for a $\sD_X[s]$-module $M$ (a module over the Weyl algebra tensored with $\C[s]$) and $x\in X$, we write $M_x$ for the stalk at $x$ of the analytification $\sO^{an}_X\otimes_{\sO_X}M$, and we put
\[
B^{a,b}_{f,x}:=\Ann_{\C[s]}(M^{a,b}_f)_x=\{b(s)\in\C[s]\mid b(s)f^{s-a}\in\sD^{an}_{X,x}[s]f^{s-b}\},\qquad B^K_{f,x}:=\Ann_{\C[s]}(N^K_f)_x .
\]

\begin{lemma}\label{lem:stalks}
Let $f$ be an $r$-tuple of polynomials on $X=\C^n$, $a\geq b$ in $\Z^r$, and $K$ a monoid ideal.
\begin{enumerate}
\item $B^{a,b}_f=\bigcap_{x\in X}B^{a,b}_{f,x}$ and $B^K_f=\bigcap_{x\in X}B^K_{f,x}$.
\item If $f$ is a complete factorization of a central arrangement, then $B^{a,b}_{f,0}\subseteq B^{a,b}_{f,x}$ and $B^K_{f,0}\subseteq B^K_{f,x}$ for every $x\in X$. Consequently $B^{a,b}_f=B^{a,b}_{f,0}$ and $B^K_f=B^K_{f,0}$.
\end{enumerate}
\end{lemma}

\begin{proof}
(1) Let $M$ be one of the modules in question. If $b(s)$ annihilates $M$, then it annihilates $\sO^{an}_X\otimes_{\sO_X}M$ and all its stalks. Conversely, if $b(s)$ annihilates every stalk $M_x$, then for $m\in M$ the section $b(s)m$ of the coherent $\sO_X$-module $\sO_X\cdot b(s)m$ vanishes in $\sO^{an}_{X,x}\otimes_{\sO_{X,x}}(\sO_X b(s)m)_x$ for every $x$, hence in $(\sO_Xb(s)m)_x$ by the faithful flatness of $\sO^{an}_{X,x}$ over $\sO_{X,x}$, hence $b(s)m=0$.

(2) Let $b(s)f^{s-a}=Pf^{s-b}$ with $P\in\sD^{an}_{X,x}[s]$ defined on a neighborhood $\Omega$ of $0$ (the case of $N^K_f$ is identical). For $t\in\C^*$ consider the $\C[s]$-linear map $\Sigma_t\colon g(x,s)f^s\mapsto g(tx,s)f^s$ on the ambient module $j_*(\sO^{an}_U[s]f^s)$. Since each $f_j$ is linear, $(\partial_{x_\alpha}f_j/f_j)(tx)=t^{-1}(\partial_{x_\alpha}f_j/f_j)(x)$, and one checks from the derivation rule $\partial_{x_\alpha}(gf^s)=(\partial_{x_\alpha}g+g\sum_j s_j\partial_{x_\alpha}f_j/f_j)f^s$ that $\Sigma_t\circ P=P^{t}\circ\Sigma_t$, where $P^t$ is obtained from $P$ by substituting $x\mapsto tx$ in the coefficients and $\partial_{x_\alpha}\mapsto t^{-1}\partial_{x_\alpha}$; the operator $P^t$ is defined on $t^{-1}\Omega$. Moreover $\Sigma_t(f^{s+c})=t^{|c|}f^{s+c}$, where $|c|=\sum_jc_j$, because $f^c$ is a polynomial of degree $|c|$. Applying $\Sigma_t$ to the functional equation gives $b(s)f^{s-a}=t^{|a|-|b|}P^tf^{s-b}$ on $t^{-1}\Omega$. Choosing $|t|$ small, $t^{-1}\Omega$ contains any given point $x$, so $b(s)\in B^{a,b}_{f,x}$. The last assertion follows from (1).
\end{proof}

\begin{theorem}[{\cite[Theorems 1.1 and 1.4]{Bat23}}]\label{thm:Bath-principal-radical}
Let $D$ be a free (more generally, tame) central arrangement with complete factorization $f$ and let $a\in\N^r\setminus\{0\}$. Then the ideal $B^{-a}_f$ is principal and radical. Consequently, by translation, $B^{a,b}_f=\tau_a(B^{-(a-b)}_f)$ is principal and radical for all $a\geq b$ in $\Z^r$ with $a\neq b$.
\end{theorem}

For a free arrangement $D$ and $i\in\{1,\dots,r\}$ we henceforth write
\begin{equation}\label{eq:bDi}
B^{-e_i}_f=(b_{D,i}(s)),\qquad b_{D,i}\ \text{squarefree, unique up to a nonzero constant.}
\end{equation}
By Lemma \ref{lem:stalks}, $b_{D,i}$ also generates $B^{-e_i}_{f,0}$.

\begin{lemma}\label{lem:unit-factors}
Let $f$ be an $r$-tuple of polynomials on $X=\C^n$, let $x\in X$, $J_x=\{j\mid f_j(x)=0\}$, and let $f_{J_x}=(f_j)_{j\in J_x}$, $s_{J_x}=(s_j)_{j\in J_x}$. Put $u^s=\prod_{j\notin J_x}f_j^{s_j}$. Then conjugation $P\mapsto u^sPu^{-s}$ is an automorphism $\phi_u$ of $\sD^{an}_{X,x}[s]$ which is the identity on $\sO^{an}_{X,x}[s]$ and sends a vector field $\vartheta$ to $\vartheta-\sum_{j\notin J_x}s_j\vartheta(f_j)/f_j$, and multiplication by $u^s$ induces, for all $a\geq b$ in $\Z^r$, isomorphisms of $\C[s]$-modules
\[
(M^{a_{J_x},b_{J_x}}_{f_{J_x}})_x\otimes_{\C[s_{J_x}]}\C[s]\xrightarrow{\ \sim\ }(M^{a,b}_f)_x ,
\]
where $M^{a_{J_x},b_{J_x}}_{f_{J_x}}$ is the module of the tuple $f_{J_x}$ over $\C[s_{J_x}]$. Consequently $B^{a,b}_{f,x}=B^{a_{J_x},b_{J_x}}_{f_{J_x},x}\cdot\C[s]$, and similarly for the modules $N^K_f$ with $K\subseteq\N^r$ replaced by its image in $\N^{J_x}$.
\end{lemma}

\begin{proof}
Since every $f_j$ with $j\notin J_x$ is a unit at $x$, $u^{\pm s}$ acts on $j_*(\sO^{an}_U[s]f^s)_x$ and $\phi_u(P)=u^sPu^{-s}$ is a well-defined differential operator with coefficients in $\sO^{an}_{X,x}[s]$ whose formula on functions and on vector fields is as stated. Thus $\phi_u$ is an automorphism of $\sD^{an}_{X,x}[s]$ fixing $\C[s]$, with $\phi_u(P)(u^sh)=u^sP(h)$. Hence multiplication by $u^s$ maps $\sD^{an}_{X,x}[s]f_{J_x}^{s_{J_x}+c_{J_x}}$ onto $\sD^{an}_{X,x}[s]f^{s+c}$ for every $c\in\Z^r$, compatibly with inclusions. Finally $\sD^{an}_{X,x}[s]f_{J_x}^{s_{J_x}+c_{J_x}}=\sD^{an}_{X,x}[s_{J_x}]f_{J_x}^{s_{J_x}+c_{J_x}}\otimes_{\C[s_{J_x}]}\C[s]$, since the variables $s_j$, $j\notin J_x$, do not occur in $f_{J_x}^{s_{J_x}+c_{J_x}}$ and $\C[s]$ is free over $\C[s_{J_x}]$; annihilators commute with this base change because a polynomial in $\C[s]=\C[s_{J_x}][s_j\mid j\notin J_x]$ annihilates $N\otimes_{\C[s_{J_x}]}\C[s]$ if and only if all its coefficients annihilate $N$.
\end{proof}

\begin{lemma}\label{lem:localization-edge}
Let $D$ be a central arrangement with complete factorization $f$, let $W\in L(D)$ and let $x\in W^\circ:=W\setminus\bigcup_{j\notin J(W,f)}D_j$. Let $i\in\{1,\dots,r\}$ and put $J=J(W,f)$. If $i\notin J$, then $(M^{-e_i}_f)_x=0$ and $B^{-e_i}_{f,x}=\C[s]$. If $i\in J$, then
\[
B^{-e_i}_{f,x}=B^{-e_i}_{f_W,0}\cdot\C[s],
\]
where $B^{-e_i}_{f_W,0}\subseteq\C[s_J]$ is the local Bernstein-Sato ideal at the origin of $X/W$ of the complete factorization $f_W$ of the essential arrangement $D^W$. Consequently, if $D$ is free then
\begin{equation}\label{eq:intersection-over-edges}
B^{-e_i}_f=\bigcap_{W\in L(D)}B^{-e_i}_{f_W,0}\cdot\C[s]=\Big(\operatorname{lcm}_{W\in L(D),\,i\in J(W,f)}b_{D^W,i}\Big),
\end{equation}
where $b_{D^W,i}\in\C[s_{J(W,f)}]$ is the generator of $B^{-e_i}_{f_W}=B^{-e_i}_{f_W,0}$ provided by Theorem \ref{thm:Bath-principal-radical} applied to the free arrangement $D^W$.
\end{lemma}

\begin{proof}
Since $J_x=J$, Lemma \ref{lem:unit-factors} reduces the statement to the tuple $f_W=(f_j)_{j\in J}$; if $i\notin J$ then $f_i$ is a unit at $x$ and $(M^{-e_i}_f)_x=0$. Assume $i\in J$ and choose linear coordinates $X=N\oplus W$ such that the $f_j$, $j\in J$, are linear forms on $N\simeq X/W$; write $x=(0,w_0)$. A functional equation $b(s)f_W^{s_J}=Pf_if_W^{s_J}$ with $P\in\sD^{an}_{N,0}[s_J]$ remains valid on $X$ near $x$, whence $B^{-e_i}_{f_W,0}\C[s]\subseteq B^{-e_i}_{f,x}$. Conversely, let $b(s)f_W^{s_J}=Pf_if_W^{s_J}$ with $P\in\sD^{an}_{X,x}[s]$. Write $P=\sum_\beta P_\beta\partial_w^\beta$ with $P_\beta\in\sD^{an}_{N\times W,x}[s]$ containing no derivatives in the $W$-directions. Since $f_W^{s_J}$ does not depend on the coordinates $w$ of $W$, the terms with $\beta\neq 0$ act by zero, and restricting the coefficients of $P_0$ to the slice $w=w_0$ gives an operator $P'\in\sD^{an}_{N,0}[s]$ with $b(s)f_W^{s_J}=P'f_if_W^{s_J}$. Expanding $b$ and $P'$ in the variables $s_j$, $j\notin J$, and comparing coefficients (the module $j_*(\sO^{an}_U[s]f^s)$ is free over $\C[s]$), we obtain $b\in B^{-e_i}_{f_W,0}\C[s]$.

The first equality in \eqref{eq:intersection-over-edges} follows from Lemma \ref{lem:stalks}(1), because $X=\bigsqcup_{W\in L(D)}W^\circ$; the second holds because in the unique factorization domain $\C[s]$ an intersection of principal ideals is generated by the least common multiple of the generators. Note that $D^W$ is free by Theorem \ref{thm:free-facts}(3) and that $B^{-e_i}_{f_W}=B^{-e_i}_{f_W,0}$ by Lemma \ref{lem:stalks}(2).
\end{proof}

The next lemma will be used to exclude components of $Z(B^{-e_i}_f)$ which are invariant under the translation $s\mapsto s+e_i$. Its proof is a noetherianity argument based on Theorem \ref{thm:Wu-jstar}.

\begin{lemma}\label{lem:ACC}
Let $f$ be an $r$-tuple of polynomials on $X=\C^n$, let $h\in\{1,\dots,r\}$ and let $\mathfrak q\subseteq\C[s]$ be a prime ideal such that $\tau_{e_h}(\mathfrak q)=\mathfrak q$, where $\tau_{e_h}(b)(s)=b(s-e_h)$. Then for every $c\in\Z^r$ and every $a\in\Z$,
\[
\sD_X[s]_{\mathfrak q}f^{s+c+ae_h}=\sD_X[s]_{\mathfrak q}f^{s+c}.
\]
In particular $(M^{-e_h}_f)_{\mathfrak q}=0$, and $\mathfrak q\not\supseteq B^{-e_h}_f$. If $Z(B^{-e_h}_f)$ is a union of hyperplanes $(L\cdot s+\alpha=0)$, then $L_h\neq 0$ for each of them.
\end{lemma}

\begin{proof}
By Theorem \ref{thm:Wu-jstar} the module $j_*(\sO_U[s]f^s)_{\mathfrak q}$ is a coherent module over the noetherian ring $\sD_X[s]_{\mathfrak q}$, and it contains the increasing chain of submodules $\sD_X[s]_{\mathfrak q}f^{s-ae_h}$, $a\in\N$. Hence $\sD_X[s]_{\mathfrak q}f^{s-ae_h}=\sD_X[s]_{\mathfrak q}f^{s-(a+1)e_h}$, that is, $(M^{(a+1)e_h,ae_h}_f)_{\mathfrak q}=0$, for all $a\gg 0$. By translation, $M^{(a+1)e_h,ae_h}_f$ is $M^{-e_h}_f$ with the $\C[s]$-structure twisted by $\tau_{(a+1)e_h}$; since $\tau_{e_h}$ preserves $\mathfrak q$, we get $(M^{-e_h}_f)_{\mathfrak q}=0$, and then $(M^{c+e_h,c}_f)_{\mathfrak q}=0$ for every $c\in\Z^r$ by the same argument. This proves the displayed equality. Finally, $(M^{-e_h}_f)_{\mathfrak q}=0$ means $\mathfrak q\notin\supp_{\C[s]}(M^{-e_h}_f)=Z(B^{-e_h}_f)$ (Lemma \ref{lem:Z=supp}), and a hyperplane $(L\cdot s+\alpha=0)$ is invariant under $s\mapsto s+e_h$ if and only if $L_h=0$.
\end{proof}

\begin{corollary}\label{cor:invertible}
In the situation of Lemma \ref{lem:ACC}, $f_h$ acts invertibly on $\sD_X[s]_{\mathfrak q}f^{s+c}$ for every $c\in\Z^r$; in other words, $\sD_X[s]_{\mathfrak q}f^{s+c}=\big(\sD_X[s]f^{s+c}\big)(*D_h)_{\mathfrak q}$, where $D_h=(f_h=0)$ and $(*D_h)$ denotes the algebraic localization along $D_h$.
\end{corollary}

\begin{proof}
Let $\sM$ be a coherent $\sD_X$-module generated by an element $m$. Since $\sO_X[1/f_h]\cdot\sD_X=\sD_X\cdot\sO_X[1/f_h]$ inside $\sD_X(*D_h)$, the localization $\sM(*D_h)=\sO_X[1/f_h]\otimes_{\sO_X}\sM$ is generated over $\sD_X$ by the elements $f_h^{-a}m$, $a\in\N$. For $\sM=\sD_X[s]f^{s+c}$ and $m=f^{s+c}$ this gives
\[
\sM(*D_h)=\bigcup_{a\in\N}\sD_X[s]f^{s+c-ae_h},
\]
and after localizing at $\mathfrak q$ all the terms of this union coincide by Lemma \ref{lem:ACC}.
\end{proof}

\section{The local multiplicity theorem}\label{sec:local}

Throughout this section $D$ is an essential, irreducible, free central hyperplane arrangement in $X=\C^n$ with $r$ hyperplanes, complete factorization $f=(f_1,\dots,f_r)$ and exponents $(1,d_2,\dots,d_n)$; by Lemma \ref{lem:exponent-one} exactly one exponent equals $1$, and by Saito's criterion $r=1+\sum_{a=2}^nd_a$, so that $\rho=r-2n+1=\sum_{a=2}^n(d_a-2)\geq 0$ as in \eqref{eq:rho}. By simultaneously relabeling the hyperplanes $D_j$, the linear forms $f_j$ and the parameters $s_j$, it suffices to prove Theorem \ref{thm:main-local} for $i=1$, which we assume from now on. We use the notation $\lambda_\nu$, $\Gamma_\nu$ of \eqref{eq:lambda-nu} and put
\begin{equation}\label{eq:mu-h}
\mu_\nu:=m_{\Gamma_\nu}(M^{-e_1}_f),\qquad h_\nu:=[t^\nu]\prod_{a=2}^n\big(1+t+\dots+t^{d_a-2}\big)\qquad(\nu\in\Z).
\end{equation}
Since every factor has positive coefficients in all degrees from $0$ to its degree, $h_\nu>0$ for $0\leq\nu\leq\rho$ and $h_\nu=0$ otherwise. Theorem \ref{thm:main-local} is the following statement.

\begin{theorem}\label{thm:local}
With notation as above, $\mu_\nu=h_\nu$ for every $\nu\in\Z$. Moreover, for $\alpha\in\C$ the linear form $\one_r\cdot s+\alpha=\sum_{j=1}^rs_j+\alpha$ divides $b_{D,1}$ if and only if $\alpha=n+\nu$ with $0\leq\nu\leq\rho$.
\end{theorem}

When $n=1$, irreducibility forces $r=1$ and $f=(x_1)$ after a change of coordinates; then $M^{-e_1}_f=\sD_X[s_1]/\sD_X[s_1](x_1\partial_{x_1}-s_1,x_1)=\Delta\otimes_\C\C[s_1]/(s_1+1)$, whose relative characteristic cycle is $[T^*_{\{0\}}X\times(s_1+1=0)]$, and $B^{-e_1}_f=(s_1+1)$. Since $\rho=0$ and $h_0=1$ (an empty product), the theorem holds in this case. \emph{We assume $n\geq 2$ for the rest of this section.}

\subsection{An adapted Saito basis}\label{sec:adapted-basis}
Choose linear coordinates $x_1,\dots,x_n$ on $X$ with $x_1=f_1$. By Lemma \ref{lem:exponent-one} there is a homogeneous basis $E,\eta_2,\dots,\eta_n$ of $\Der(-\log D)$ with $\deg\eta_a=d_a\geq 2$. Since $\eta_a$ is logarithmic along $D_1$, the quotient $\eta_a(x_1)/x_1$ is a polynomial, and we replace $\eta_a$ by
\begin{equation}\label{eq:theta-a}
\theta_a:=\eta_a-\frac{\eta_a(x_1)}{x_1}\,E\qquad(2\leq a\leq n).
\end{equation}
This triangular change of basis has determinant one, so $E,\theta_2,\dots,\theta_n$ is again a homogeneous basis with $\deg\theta_a=d_a$, and $\theta_a(x_1)=0$. Since the coefficient of $\partial_{x_1}$ in $\theta_a$ is $\theta_a(x_1)$, each $\theta_a$ has no $\partial_{x_1}$-component. We write
\begin{equation}\label{eq:theta-coeff}
\theta_a=\sum_{\alpha=2}^n p_{a\alpha}(x)\,\partial_{x_\alpha},\qquad \varrho_{aj}:=\frac{\theta_a(f_j)}{f_j}\qquad(2\leq a\leq n,\ 1\leq j\leq r);
\end{equation}
thus $\varrho_{a1}=0$, and $p_{a\alpha}$, $\varrho_{aj}$ are homogeneous of degrees $d_a$ and $d_a-1$. Let $H=D_1=(x_1=0)$, and denote by a bar the restriction to $H$. By Theorem \ref{thm:Ziegler}(2),
\begin{equation}\label{eq:Ziegler-basis}
\ol\theta_2,\dots,\ol\theta_n\ \text{ is a homogeneous basis of } \Der(-\log(D|_H,m_H)),\quad \ol\theta_a(\ol f_j)=\ol\varrho_{aj}\ol f_j\quad(j\neq 1).
\end{equation}

\subsection{A generic complete intersection for the Ziegler restriction}\label{sec:generic-ci}
This subsection constructs the Artinian complete intersection whose Hilbert function is the right-hand side of \eqref{eq:main-local}. Let
\[
R_H:=\C[H]\simeq\C[x_2,\dots,x_n],\qquad H_j:=D_j\cap H=(\ol f_j=0)\quad(j\neq 1);
\]
different indices $j$ may give the same hyperplane $H_j$ of $H$, and these repetitions are exactly what the Ziegler multiplicities record. Let
\[
\Omega:=\{\omega=(\omega_j)_{j\neq 1}\mid\omega_j\in\C\}\simeq\C^{r-1}
\]
be the \emph{weight space}, and for $\omega\in\Omega$ define
\begin{equation}\label{eq:g-a}
g_a(\omega):=\sum_{j\neq 1}\omega_j\,\ol\varrho_{aj}\in(R_H)_{d_a-1}\qquad(2\leq a\leq n),
\end{equation}
where $(R_H)_q$ denotes the homogeneous part of degree $q$. The same weight vector is used for every $a$, which is why the genericity statement below is not immediate. The idea is the following: if a point $[x]\in\PP(H)$ moves in a stratum whose carrier flat has dimension $\ell$, it has $\ell-1$ degrees of freedom, while the vanishing of all $g_a(\omega)(x)$ imposes at least $\ell$ independent linear conditions on $\omega$; hence a general weight is bad nowhere.

A \emph{flat} of the simple arrangement $D|_H$ is an intersection of some of its hyperplanes (the empty intersection being $H$). For $0\neq x\in H$ the \emph{carrier flat} of $x$ is
\begin{equation}\label{eq:carrier}
Y_x:=\bigcap_{H'\in D|_H,\ x\in H'}H',\qquad \ell_x:=\dim Y_x ,
\end{equation}
and for a nonzero flat $Y$ the \emph{stratum} of $Y$ is
\begin{equation}\label{eq:stratum}
S_Y:=\{[x]\in\PP(H)\mid Y_x=Y\}=\PP(Y)\setminus\bigcup_{H'\in D|_H,\ Y\not\subseteq H'}\PP(Y\cap H'),
\end{equation}
a Zariski open subset of $\PP(Y)$, so that $\dim S_Y=\dim Y-1$. The finitely many strata $S_Y$ are pairwise disjoint and cover $\PP(H)$. The restricted arrangement is essential:
\begin{equation}\label{eq:restriction-essential}
\bigcap_{H'\in D|_H}H'=\bigcap_{j\neq 1}(H\cap D_j)=\bigcap_{j=1}^rD_j=\{0\},
\end{equation}
because $H=D_1$ and $D$ is essential.

\begin{lemma}\label{lem:Ziegler-values}
For every $0\neq x\in H$, the evaluation at $x$ of the stalk at $x$ of the localized Ziegler module satisfies $\ev_x\big(\Der(-\log(D|_H,m_H))_x\big)=T_xY_x\subseteq T_xH\simeq H$. Consequently
\begin{equation}\label{eq:span-theta}
\Span_\C\{\ol\theta_2(x),\dots,\ol\theta_n(x)\}=Y_x .
\end{equation}
\end{lemma}

\begin{proof}
Localization is exact, so the stalk at $x$ consists of the derivations over $\sO_{H,x}$ satisfying the divisibility conditions defining $\Der(-\log(D|_H,m_H))$. Let $\eta$ be such a derivation. If $H'=(\beta_{H'}=0)\in D|_H$ contains $x$, then $\eta(\beta_{H'})\in\beta_{H'}^{m_H(H')}\sO_{H,x}$ vanishes at $x$, so $\eta(x)\in T_xH'$; hence $\eta(x)\in T_xY_x$. Conversely let $v\in Y_x$ and let $\partial_v$ be the constant vector field in the direction $v$. If $x\in H'$ then $v\in Y_x\subseteq H'$ and $\partial_v(\beta_{H'})=\beta_{H'}(v)=0$; if $x\notin H'$ then $\beta_{H'}$ is a unit at $x$ and the divisibility condition along $H'$ is automatic. Thus $\partial_v$ lies in the stalk, which proves the first assertion. Finally, $\ol\theta_2,\dots,\ol\theta_n$ remain a basis after localization by \eqref{eq:Ziegler-basis}, so the value at $x$ of every local derivation is a complex linear combination of $\ol\theta_2(x),\dots,\ol\theta_n(x)$, which gives \eqref{eq:span-theta}.
\end{proof}

Fix $0\neq x\in H$. Evaluation of the $g_a$ at $x$ defines a linear map in the weight variables
\begin{equation}\label{eq:Psi-x}
\Psi_x\colon\Omega\to\C^{n-1},\qquad\omega\mapsto\big(g_2(\omega)(x),\dots,g_n(\omega)(x)\big).
\end{equation}
Replacing $x$ by $cx$, $c\in\C^*$, multiplies the $a$-th component by $c^{d_a-1}$, so $\ker\Psi_x$ and $\rank\Psi_x$ only depend on $[x]\in\PP(H)$. Put
\[
J^0_x:=\{j\neq 1\mid \ol f_j(x)=0\},\qquad J^\times_x:=\{j\neq 1\mid\ol f_j(x)\neq 0\},
\]
\[
\Omega_x:=\{\omega\in\Omega\mid\omega_j=0\text{ for } j\in J^0_x\}.
\]
For $j\in J^\times_x$ the covector $\alpha_{j,x}:=(d\ol f_j/\ol f_j(x))|_{T_xY_x}\in T^*_xY_x$ is well defined; it is the normal covector of $H_j$ restricted to $Y_x$ and normalized by $\ol f_j(x)$. Clearly $\rank\Psi_x\geq\rank(\Psi_x|_{\Omega_x})$. Define
\begin{align*}
\mathsf N_x\colon\Omega_x\to T^*_xY_x,\qquad&\mathsf N_x(\omega)=\sum_{j\in J^\times_x}\omega_j\alpha_{j,x},\\
\mathsf E_x\colon T^*_xY_x\to\C^{n-1},\qquad&\mathsf E_x(\alpha)=\big(\alpha(\ol\theta_2(x)),\dots,\alpha(\ol\theta_n(x))\big).
\end{align*}

\begin{lemma}\label{lem:rank}
The map $\mathsf N_x$ is surjective, the map $\mathsf E_x$ is injective, and $\Psi_x|_{\Omega_x}=\mathsf E_x\circ\mathsf N_x$. Consequently $\rank\Psi_x\geq\ell_x=\dim Y_x$.
\end{lemma}

\begin{proof}
Let $v\in Y_x$ be annihilated by every covector in the image of $\mathsf N_x$. Taking one weight coordinate at a time we get $d\ol f_j(v)=\ol f_j(v)=0$ for all $j\in J^\times_x$; for $j\in J^0_x$ we have $v\in Y_x\subseteq H_j$, so $\ol f_j(v)=0$ as well. Hence $v\in\bigcap_{j\neq 1}H_j=\{0\}$ by \eqref{eq:restriction-essential}. Thus the annihilator of $\im\mathsf N_x$ in $T_xY_x$ is zero and $\mathsf N_x$ is surjective (repeated indices give equal covectors and do not affect the span). By Lemma \ref{lem:Ziegler-values} the vectors $\ol\theta_a(x)$ span $Y_x$, so a covector killed by $\mathsf E_x$ vanishes on $T_xY_x$; thus $\mathsf E_x$ is injective. For $\omega\in\Omega_x$ the $a$-th component of $\mathsf E_x\mathsf N_x(\omega)$ is
\[
\sum_{j\in J^\times_x}\omega_j\frac{d\ol f_j(\ol\theta_a(x))}{\ol f_j(x)}=\sum_{j\in J^\times_x}\omega_j\frac{\ol\theta_a(\ol f_j)(x)}{\ol f_j(x)}=\sum_{j\in J^\times_x}\omega_j\ol\varrho_{aj}(x)=g_a(\omega)(x),
\]
by \eqref{eq:Ziegler-basis} and because $\omega_j=0$ for $j\in J^0_x$. A surjection onto $T^*_xY_x$ followed by an injection has rank $\dim Y_x$, so $\rank(\Psi_x|_{\Omega_x})=\ell_x$ and $\rank\Psi_x\geq\ell_x$.
\end{proof}


\begin{lemma}\label{lem:generic-regular}
For a Zariski general $\omega\in\Omega$, the polynomials $g_2(\omega),\dots,g_n(\omega)$ have no common zero in $\PP(H)$. Consequently they form a homogeneous regular sequence in $R_H$, and the quotient $R_H/(g_2(\omega),\dots,g_n(\omega))$ is an Artinian complete intersection.
\end{lemma}

\begin{proof}
Fix a nonzero flat $Y$, put $\ell=\dim Y$, and consider the incidence set
\[
\sI_Y:=\{([x],\omega)\in S_Y\times\Omega\mid g_a(\omega)(x)=0\text{ for all }2\leq a\leq n\},
\]
which is well defined because the equations are homogeneous in $x$. The fiber of $\sI_Y$ over $[x]\in S_Y$ is $\ker\Psi_x$, of dimension at most $(r-1)-\ell$ by Lemma \ref{lem:rank}, since $Y_x=Y$. As $\dim S_Y=\ell-1$, the fiber dimension inequality gives $\dim\sI_Y\leq(\ell-1)+(r-1-\ell)=r-2<r-1=\dim\Omega$. Hence the closure of the projection of $\sI_Y$ to $\Omega$ is a proper closed subset of the irreducible variety $\Omega$, and so is the union of these closed subsets over the finitely many flats $Y$. A weight $\omega$ outside this union is Zariski general, and for such $\omega$ there is no projective point at which all $g_a(\omega)$ vanish.

Put $I_\omega=(g_2(\omega),\dots,g_n(\omega))\subseteq R_H$ and $\mathfrak m_H=(x_2,\dots,x_n)$. Every $g_a$ has positive degree because $d_a\geq 2$, so the origin is a common affine zero; since there is no common projective zero, the common zero set of $I_\omega$ in $H$ is $\{0\}$ and $\sqrt{I_\omega}=\mathfrak m_H$ by the homogeneous Nullstellensatz. Thus $I_\omega$ is $\mathfrak m_H$-primary and $R_H/I_\omega$ is finite dimensional over $\C$. The ring $R_H$ has dimension $n-1$ and $I_\omega$ is generated by $n-1$ homogeneous elements, which therefore form a homogeneous system of parameters; since $R_H$ is Cohen-Macaulay, they form a regular sequence (see \cite[Chapters 17 and 18]{Eis95}).
\end{proof}

Fix a general weight $\omega$ as in Lemma \ref{lem:generic-regular}, abbreviate $g_a=g_a(\omega)$, and put $\epsilon_a:=d_a-1=\deg g_a$ and
\begin{equation}\label{eq:Q-omega}
Q_\omega:=R_H/(g_2,\dots,g_n).
\end{equation}
For a graded algebra $Q=\bigoplus_{\nu\geq 0}Q_\nu$ with finite-dimensional graded pieces let $\Hilb_Q(t)=\sum_\nu(\dim_\C Q_\nu)t^\nu$. Since $g_2,\dots,g_n$ is a regular sequence, multiplication by $g_a$ is injective on $R_H/(g_2,\dots,g_{a-1})$, and the resulting exact sequences $0\to Q'(-\epsilon_a)\xrightarrow{g_a}Q'\to Q'/g_aQ'\to 0$ give
\begin{equation}\label{eq:Hilbert-series}
\Hilb_{Q_\omega}(t)=\frac{\prod_{a=2}^n(1-t^{d_a-1})}{(1-t)^{n-1}}=\prod_{a=2}^n\big(1+t+\dots+t^{d_a-2}\big)=\sum_{\nu}h_\nu t^\nu ,
\end{equation}
with $h_\nu$ as in \eqref{eq:mu-h}. In particular $\dim_\C(Q_\omega)_\nu=h_\nu$ does not depend on the choice of the general weight, and $Q_\omega$ lives in degrees $0\leq\nu\leq\rho$.

\subsection{Localization at a parameter hyperplane}\label{sec:localization-parameter}
For $\alpha\in\C$ let $\mathfrak q_\alpha:=(\one_r\cdot s+\alpha)\subseteq\C[s]$ and $\K_\alpha:=\Frac(\C[s]/\mathfrak q_\alpha)$, the residue field of the discrete valuation ring $\C[s]_{\mathfrak q_\alpha}$. For $\nu\in\Z$ we have $\mathfrak q_{n+\nu}=(\lambda_\nu)$, and we write $\mathfrak q_\nu$, $\K_\nu$ instead of $\mathfrak q_{n+\nu}$, $\K_{n+\nu}$ when no confusion is possible.

\begin{lemma}\label{lem:support-origin}
Let $\mathfrak q\subseteq\C[s]$ be a prime ideal such that $\mathfrak q\cap\C[s_J]=0$ for every proper subset $J\subsetneq\{1,\dots,r\}$, where $\C[s_J]=\C[s_j\mid j\in J]$. Then the module $(M^{-e_1}_f)_{\mathfrak q}$ is supported at the origin: its analytic stalks vanish at every $x\neq 0$, and every element of $(M^{-e_1}_f)_{\mathfrak q}$ is annihilated by a power of $(x_1,\dots,x_n)$. The hypothesis holds for $\mathfrak q=\mathfrak q_\alpha$, $\alpha\in\C$.
\end{lemma}

\begin{proof}
Let $x\neq 0$ and let $J_x=\{j\mid f_j(x)=0\}$, a proper subset of $\{1,\dots,r\}$ because $D$ is essential. If $1\notin J_x$ then $f_1$ is a unit at $x$ and $(M^{-e_1}_f)_x=0$. If $1\in J_x$, then by Theorem \ref{thm:Sabbah} applied to the germ of the tuple $f_{J_x}=(f_j)_{j\in J_x}$ at $x$ there is a nonzero $b_x\in\C[s_{J_x}]$ with $b_xf_{J_x}^{s_{J_x}}\in\sD^{an}_{X,x}[s_{J_x}]f_{J_x}^{s_{J_x}+\one_{J_x}}\subseteq\sD^{an}_{X,x}[s_{J_x}]f_{J_x}^{s_{J_x}+e_1}$, so that $b_x$ annihilates $(M^{-e_1}_{f_{J_x}})_x$, hence $(M^{-e_1}_f)_x$ by Lemma \ref{lem:unit-factors}. By hypothesis $b_x\notin\mathfrak q$, so $b_x$ becomes a unit in $\C[s]_{\mathfrak q}$ and $(M^{-e_1}_f)_{x,\mathfrak q}=0$. For $m\in(M^{-e_1}_f)_{\mathfrak q}$ the coherent $\sO_X$-module generated by $m$ has vanishing analytic stalks away from $0$, hence is supported at $0$ by faithful flatness, so $m$ is annihilated by a power of $(x_1,\dots,x_n)$. Finally, if $J$ is a proper subset and $h\notin J$, then every nonzero element of $\mathfrak q_\alpha=(\one_r\cdot s+\alpha)$ has positive degree in $s_h$ because $\C[s]$ is a domain; hence $\mathfrak q_\alpha\cap\C[s_J]=0$.
\end{proof}

By Theorem \ref{thm:Bath-principal-radical} and \eqref{eq:bDi}, $B^{-e_1}_f=(b_{D,1})$ with $b_{D,1}$ squarefree. In the discrete valuation ring $\C[s]_{\mathfrak q_\alpha}$ the element $b_{D,1}$ is therefore either a unit or a unit times $\one_r\cdot s+\alpha$; equivalently,
\begin{equation}\label{eq:zero-or-killed}
(M^{-e_1}_f)_{\mathfrak q_\alpha}=0\qquad\text{or}\qquad(\one_r\cdot s+\alpha)\cdot(M^{-e_1}_f)_{\mathfrak q_\alpha}=0 ,
\end{equation}
and the first alternative holds if and only if $\one_r\cdot s+\alpha$ does not divide $b_{D,1}$. In either case $(M^{-e_1}_f)_{\mathfrak q_\alpha}$ is a cyclic module over $\sD_X\otimes_\C\K_\alpha$ which is supported at the origin by Lemma \ref{lem:support-origin}. Lemma \ref{lem:kashiwara}(3) therefore gives
\begin{equation}\label{eq:M-is-delta}
(M^{-e_1}_f)_{\mathfrak q_\alpha}\simeq\delta_{0,\K_\alpha}^{\oplus m(\alpha)},\qquad m(\alpha):=\dim_{\K_\alpha}\Hom_{\sD_X\otimes\K_\alpha}\big((M^{-e_1}_f)_{\mathfrak q_\alpha},\delta_{0,\K_\alpha}\big),
\end{equation}
and $(M^{-e_1}_f)_{\mathfrak q_\alpha}\neq 0$ if and only if $m(\alpha)>0$. Moreover, when $\alpha=n+\nu$ with $\nu\in\Z$,
\begin{equation}\label{eq:m-equals-mu}
m(n+\nu)=\mu_\nu .
\end{equation}
Indeed, a relative good filtration of $M^{-e_1}_f$ localizes to a relative good filtration of $(M^{-e_1}_f)_{\mathfrak q_\nu}$, whose associated graded module is annihilated by $\lambda_\nu$ and is therefore a good filtration of $(M^{-e_1}_f)_{\mathfrak q_\nu}$ as a $\sD_X\otimes\K_\nu$-module. By \eqref{eq:localize-CC} the multiplicity of $\Gamma_\nu$ in $\CCrel(M^{-e_1}_f)$, which is the length of the graded module at the generic point of $\Gamma_\nu=T^*_{\{0\}}X\times(\lambda_\nu=0)$, equals the multiplicity of $T^*_{\{0\}}X\otimes\K_\nu$ in $\CC((M^{-e_1}_f)_{\mathfrak q_\nu})$, which is $m(n+\nu)$ by \eqref{eq:M-is-delta} and Lemma \ref{lem:kashiwara}(3).

\subsection{Estimate of local multiplicities}\label{sec:upper-bound}
By Theorem \ref{thm:log-ann} applied to the adapted basis $E,\theta_2,\dots,\theta_n$ and Lemma \ref{lem:cyclic},
\begin{equation}\label{eq:cyclic-presentation-unit}
M^{-e_1}_f\simeq\frac{\sD_X[s]}{\sD_X[s]\big(x_1,\wt E,\wt\theta_2,\dots,\wt\theta_n\big)},\qquad \wt E=E-\one_r\cdot s,\qquad\wt\theta_a=\theta_a-\sum_{j=1}^rs_j\varrho_{aj},
\end{equation}
since $E(f_j)=f_j$ for every $j$. Fix $\alpha\in\C$ and write $\K=\K_\alpha$. A $\sD_X\otimes\K$-linear map $(M^{-e_1}_f)_{\mathfrak q_\alpha}\to\delta_{0,\K}$ is determined by the image $v\in\delta_{0,\K}$ of the class of $1$, and $v$ can be prescribed arbitrarily subject to
\begin{equation}\label{eq:solution-conditions}
x_1v=0,\qquad\wt Ev=0,\qquad\wt\theta_av=0\quad(2\leq a\leq n),
\end{equation}
where $s_j$ acts on $\delta_{0,\K}$ through its image in $\K$. By \eqref{eq:delta-rules}, $E$ acts on $\partial^\beta\delta_0$ by the scalar $-(n+|\beta|)$, while $\one_r\cdot s=-\alpha$ in $\K$; hence $\wt E$ acts on $\partial^\beta\delta_0$ by $\alpha-n-|\beta|$. Consequently the conditions \eqref{eq:solution-conditions} have no nonzero solution unless $\alpha-n\in\N$, and by \eqref{eq:M-is-delta}:
\begin{equation}\label{eq:non-integral}
\text{if }\alpha\notin n+\N,\ \text{then } (M^{-e_1}_f)_{\mathfrak q_\alpha}=0 \text{ and } \one_r\cdot s+\alpha\nmid b_{D,1}.
\end{equation}

From now on let $\alpha=n+\nu$ with $\nu\in\N$, so $\K=\K_\nu$. Put $R:=R_H=\C[x_2,\dots,x_n]$, with $R_q$ its degree-$q$ part ($R_q=0$ for $q<0$), and
\begin{equation}\label{eq:V-q}
V_q:=\{v\in\delta_{0,\K}\mid x_1v=0,\ Ev=-(n+q)v\}\qquad(q\in\Z),
\end{equation}
so that $V_q=0$ for $q<0$ and the elements $\partial_{x_2}^{\beta_2}\cdots\partial_{x_n}^{\beta_n}\delta_0$ with $\beta_2+\dots+\beta_n=q$ form a $\K$-basis of $V_q$ for $q\geq 0$. For multi-indices $\alpha,\beta\in\N^{n-1}$ (indexed by $2,\dots,n$) the pairing
\begin{equation}\label{eq:pairing}
\big\langle x^\alpha,\partial^\beta\delta_0\big\rangle:=(-1)^q\beta!\,\delta_{\alpha\beta}\qquad(|\alpha|=|\beta|=q)
\end{equation}
identifies $V_q$ with $\K\otimes_\C R_q^\vee$, and by \eqref{eq:delta-rules} it satisfies
\begin{equation}\label{eq:pairing-transpose}
\langle h,pv\rangle=\langle ph,v\rangle\qquad(p\in R_e,\ h\in R_{q-e},\ v\in V_q).
\end{equation}
Thus multiplication by a polynomial $p\in R_e$ on $\delta_{0,\K}$, restricted to $V_q\to V_{q-e}$, is the transpose $m_p^\vee$ of the multiplication map $m_p\colon R_{q-e}\to R_q$.

Now let $v$ satisfy \eqref{eq:solution-conditions}. The condition $\wt Ev=0$ forces $v\in V_\nu$. For $2\leq a\leq n$, the operator $\theta_a=\sum_{\alpha\geq 2}p_{a\alpha}\partial_{x_\alpha}$ preserves $\ker(x_1)$ and lowers the eigenvalue degree $q$ by $d_a-1$, and so does each $\varrho_{aj}$; moreover, every coefficient divisible by $x_1$ acts by zero on $\ker(x_1)$, so only the restrictions $\ol p_{a\alpha}$, $\ol\varrho_{aj}$ to $H$ matter. Since $\varrho_{a1}=0$, the operator $\wt\theta_a$ induces $\K$-linear maps
\begin{equation}\label{eq:T-a-nu}
T_{a,\nu}(s):=\wt\theta_a\big|_{V_\nu}=\ol\theta_a\big|_{V_\nu}-\sum_{j\neq 1}s_j\,m^\vee_{\ol\varrho_{aj}}\colon V_\nu\to V_{\nu-\epsilon_a},\qquad\epsilon_a=d_a-1 ,
\end{equation}
where $\ol\theta_a$ acts on $\ker(x_1)$ through its coefficients $\ol p_{a\alpha}$. The entries of $T_{a,\nu}(s)$ in the monomial bases are affine-linear functions of $s$ with complex coefficients; we regard $T_\nu(s):=(T_{2,\nu}(s),\dots,T_{n,\nu}(s))\colon V_\nu\to\bigoplus_{a=2}^nV_{\nu-\epsilon_a}$ as a matrix-valued polynomial function on the hyperplane $(\lambda_\nu=0)\subseteq\C^r$, whose function field is $\K_\nu$. By \eqref{eq:M-is-delta}, \eqref{eq:m-equals-mu} and the discussion above,
\begin{equation}\label{eq:mu-is-kernel}
\mu_\nu=\dim_{\K_\nu}\ker T_\nu ,
\end{equation}
and $\rank_{\K_\nu}T_\nu$ is the rank of $T_\nu(s)$ at a general point of the hyperplane $(\lambda_\nu=0)$.

Choose a general weight $\omega\in\Omega$ as in Lemma \ref{lem:generic-regular} and extend it to $\hat\omega\in\C^r$ by $\hat\omega_j=\omega_j$ for $j\neq 1$ and $\hat\omega_1=-\sum_{j\neq 1}\omega_j$. Then $\one_r\cdot\hat\omega=0$, so for $s_0\in(\lambda_\nu=0)$ the line $s(\tau)=s_0+\tau\hat\omega$ stays in the hyperplane $(\lambda_\nu=0)$. Because $\varrho_{a1}=0$, \eqref{eq:g-a} and \eqref{eq:T-a-nu} give
\begin{equation}\label{eq:perturbation}
T_{a,\nu}(s(\tau))=T_{a,\nu}(s_0)-\tau\,m^\vee_{g_a(\omega)}\qquad(2\leq a\leq n).
\end{equation}
Let
\begin{equation}\label{eq:G-nu}
G_\nu\colon\bigoplus_{a=2}^nR_{\nu-\epsilon_a}\to R_\nu,\qquad(q_a)_a\mapsto\sum_{a=2}^ng_a(\omega)q_a ,
\end{equation}
so that, under \eqref{eq:pairing}, the leading coefficient in $\tau$ of $T_\nu(s(\tau))$ is $-G_\nu^\vee$. Put $r_0=\rank_\C G_\nu^\vee=\rank_\C G_\nu$. If $r_0>0$, choose a nonzero $r_0\times r_0$ minor of $G_\nu^\vee$; the corresponding minor of $T_\nu(s(\tau))$ is a polynomial in $\tau$ whose coefficient of $\tau^{r_0}$ is $\pm$ that minor, so it does not vanish identically on $(\lambda_\nu=0)$ and is nonzero in $\K_\nu$. Hence $\rank_{\K_\nu}T_\nu\geq r_0$, and since $\dim_{\K_\nu}V_\nu=\dim_\C R_\nu$,
\[
\mu_\nu=\dim_{\K_\nu}V_\nu-\rank_{\K_\nu}T_\nu\leq\dim_\C R_\nu-\rank_\C G_\nu=\dim_\C\coker G_\nu=\dim_\C(Q_\omega)_\nu=h_\nu
\]
by \eqref{eq:Hilbert-series}. Together with \eqref{eq:non-integral} (for $\nu<0$, where $h_\nu=0$) we have proved:
\begin{equation}\label{eq:upper-bound}
\mu_\nu\leq h_\nu\qquad\text{for every }\nu\in\Z .
\end{equation}

\subsection{The total multiplicity}\label{sec:total}
For $c\in\{1,\dots,r\}$ and $\nu\in\Z$ put $\mu_{c,\nu}:=m_{\Gamma_\nu}(M^{-e_c}_f)$, so that $\mu_{1,\nu}=\mu_\nu$. Since $\Chrel(M^{-e_c}_f)$ has finitely many components, there is an integer $k_0$ with $\mu_{c,\nu}=0$ whenever $|\nu|\geq k_0$. Choose $k>\max(k_0,n)$ large enough for Theorem \ref{thm:Wu-dense} to apply with $W=\{0\}$ and $l=n$, and put $k_r=(k,\dots,k)$. Consider the coordinatewise decreasing path
\[
a^{(0)}=k_r\geq a^{(1)}\geq\dots\geq a^{(2kr)}=-k_r
\]
in $\Z^r$ built in three blocks: (1) decrement every coordinate $j\neq 1$ one unit at a time, $k$ times each, ending at $ke_1$; (2) decrement the first coordinate $2k$ times, ending at $-ke_1$; (3) decrement every coordinate $j\neq 1$ one unit at a time, $k$ times each, ending at $-k_r$. If $c_m$ denotes the coordinate decremented at step $m$, so that $a^{(m+1)}=a^{(m)}-e_{c_m}$, then the short exact sequences
\begin{equation}\label{eq:path-ses}
0\to M^{a^{(m+1)},-k_r}_f\to M^{a^{(m)},-k_r}_f\to M^{a^{(m)},a^{(m+1)}}_f\to 0
\end{equation}
show, by Theorem \ref{thm:Wu-basic}(3) applied repeatedly, that
\begin{equation}\label{eq:additivity-path}
\CCrel_{n+r-1}(M^{k_r,-k_r}_f)=\sum_{m}\CCrel_{n+r-1}\big(M^{a^{(m)},a^{(m)}-e_{c_m}}_f\big).
\end{equation}
By translation, $M^{a^{(m)},a^{(m)}-e_{c_m}}_f$ is $M^{-e_{c_m}}_f$ with the parameters shifted by $a^{(m)}$, and the translate of $\Gamma_\nu$ by $a^{(m)}$ is $T^*_{\{0\}}X\times(\one_r\cdot s+n+\nu-\one_r\cdot a^{(m)}=0)$. Hence the multiplicity of the component
\[
\Gamma:=\Gamma_0=T^*_{\{0\}}X\times(\one_r\cdot s+n=0),
\]
which has dimension $n+r-1$, in $\CCrel(M^{a^{(m)},a^{(m)}-e_{c_m}}_f)$ equals $\mu_{c_m,\,\one_r\cdot a^{(m)}}$. The sums $\one_r\cdot a^{(m)}$ at the beginning of the steps are $kr,kr-1,\dots,k+1$ in block (1), $k,k-1,\dots,-k+1$ in block (2), and $-k,-k-1,\dots,-kr+1$ in block (3). By the choice of $k$ the steps of blocks (1) and (3) contribute nothing, and every integer $\nu$ with $\mu_{1,\nu}\neq 0$ occurs exactly once in block (2). Therefore \eqref{eq:additivity-path} gives
\begin{equation}\label{eq:total-mult}
m_\Gamma\big(M^{k_r,-k_r}_f\big)=\sum_{\nu\in\Z}\mu_\nu .
\end{equation}
On the other hand, Theorem \ref{thm:Wu-dense} with $W=\{0\}$ and $l=n$ computes the left-hand side of \eqref{eq:total-mult} as $(-1)^{n-1}\chi(U(D))$, where $U(D)=\PP(X)\setminus\bigcup_{H\in D}\PP(H)$. We evaluate this number. Let $M(D)=X\setminus\bigcup_{H\in D}H$. Every projective point of $U(D)$ has a unique representative $y$ with $f_1(y)=1$, so $M(D)\to\C^*\times U(D)$, $x\mapsto(f_1(x),[x])$, is an isomorphism with inverse $(c,[y])\mapsto cy$. Hence $\pi(M(D),t)=(1+t)\pi(U(D),t)$, and Terao's factorization theorem (Theorem \ref{thm:free-facts}(4)) gives $\pi(U(D),t)=\prod_{a=2}^n(1+d_at)$. Therefore
\begin{equation}\label{eq:chi-U}
(-1)^{n-1}\chi(U(D))=(-1)^{n-1}\pi(U(D),-1)=\prod_{a=2}^n(d_a-1)=\sum_{\nu\in\Z}h_\nu ,
\end{equation}
the last equality by \eqref{eq:Hilbert-series} at $t=1$. Combining \eqref{eq:total-mult}, \eqref{eq:chi-U} and \eqref{eq:upper-bound} we obtain $\sum_\nu\mu_\nu=\sum_\nu h_\nu$ with $\mu_\nu\leq h_\nu$ for every $\nu$, hence
\begin{equation}\label{eq:equality}
\mu_\nu=h_\nu\qquad\text{for every }\nu\in\Z .
\end{equation}

\begin{proof}[Proof of Theorem \ref{thm:local}]
The first assertion is \eqref{eq:equality}. For the second, let $\alpha\in\C$. If $\alpha\notin n+\N$, then $\one_r\cdot s+\alpha\nmid b_{D,1}$ by \eqref{eq:non-integral}. Let $\alpha=n+\nu$ with $\nu\in\N$. If $0\leq\nu\leq\rho$, then $\mu_\nu=h_\nu>0$, so $\Gamma_\nu$ is a component of $\Chrel(M^{-e_1}_f)$, and $(\lambda_\nu=0)=p_2(\Gamma_\nu)\subseteq Z(B^{-e_1}_f)$ by Theorem \ref{thm:Wu-basic}(2); thus $\lambda_\nu$ divides $b_{D,1}$. If $\nu>\rho$, then $m(n+\nu)=\mu_\nu=h_\nu=0$ by \eqref{eq:m-equals-mu} and \eqref{eq:equality}, so $(M^{-e_1}_f)_{\mathfrak q_\nu}=0$ by \eqref{eq:M-is-delta}, and $\lambda_\nu$ does not divide $b_{D,1}$ by \eqref{eq:zero-or-killed}.
\end{proof}

\begin{remark}\label{rem:two-multiplicities}
The integer $\mu_\nu$ is a multiplicity of a relative characteristic cycle and may well exceed one.  It is not the exponent of $\lambda_\nu$ in $b_{D,1}$, which is one because $B^{-e_1}_f$ is radical.
\end{remark}

\begin{remark}
The proof shows slightly more than stated: for every $\alpha\in\C$, the module $(M^{-e_1}_f)_{\mathfrak q_\alpha}$ is isomorphic to $\delta_{0,\K_\alpha}^{\oplus h_{\alpha-n}}$, where $h_{\alpha-n}:=0$ if $\alpha-n\notin\Z$. In particular the ``fiber'' of $M^{-e_1}_f$ at the generic point of the hyperplane $(\lambda_\nu=0)$ is a point-supported module of length $h_\nu$.
\end{remark}

\section{Proofs of Theorem \texorpdfstring{\ref{thm:main-unit}}{1.1} and Corollary \texorpdfstring{\ref{cor:general-shift}}{1.2}}\label{sec:global}

Let $D$ be a free central arrangement in $X=\C^n$ with complete factorization $f$ and let $i\in\{1,\dots,r\}$. By Theorem \ref{thm:Bath-principal-radical} we have $B^{-e_i}_f=(b_{D,i})$ with $b_{D,i}$ squarefree, and by Lemma \ref{lem:localization-edge},
\begin{equation}\label{eq:lcm}
b_{D,i}=\operatorname{lcm}_{W\in L(D),\ i\in J(W,f)}\ b_{D^W,i},\qquad b_{D^W,i}\in\C[s_{J(W,f)}],
\end{equation}
where $b_{D^W,i}$ generates the Bernstein-Sato ideal $B^{-e_i}_{f_W}$ of the complete factorization $f_W$ of the essential free arrangement $D^W$ in $X/W$, which is a squarefree polynomial in the variables $s_j$, $j\in J(W,f)$. For a dense edge $W$ with $i\in J(W,f)$ and $0\leq\nu\leq|J(W,f)|-2\rank(W)+1$ we write
\begin{equation}\label{eq:LWnu}
L_{W,\nu}:=\sum_{j\in J(W,f)}s_j+\rank(W)+\nu ,
\end{equation}
so that $P_{D,i}=\prod_{(W,\nu)}L_{W,\nu}$. Note that for an edge $W$ the edges of $D^W$ are the subspaces $W'/W$ with $W'\in L(D)$, $W'\supseteq W$; that $J(W'/W,f_W)=J(W',f)$ and $\rank_{X/W}(W'/W)=\rank_X(W')$; and that $(D^W)^{W'/W}=D^{W'}$, so that $W'/W$ is a dense edge of $D^W$ if and only if $W'$ is a dense edge of $D$.

\subsubsection*{The lower bound}
Let $W$ be a dense edge with $i\in J(W,f)$ and put $J=J(W,f)$, $n_W=\rank(W)$, $r_W=|J|$. The arrangement $D^W$ in $X/W\simeq\C^{n_W}$ is essential, irreducible (because $W$ is dense) and free (Theorem \ref{thm:free-facts}(3)), with $r_W$ hyperplanes. Theorem \ref{thm:local}, applied to $D^W$ with the parameters $s_J$ and with the index $i$ in place of $1$, shows that $L_{W,\nu}$ divides $b_{D^W,i}$ for $0\leq\nu\leq r_W-2n_W+1$, hence divides $b_{D,i}$ by \eqref{eq:lcm}. The linear forms $L_{W,\nu}$ are pairwise non-associate: the coefficient vector of $L_{W,\nu}$ is the indicator vector of $J(W,f)$, which determines $W=\bigcap_{j\in J(W,f)}D_j$, and for fixed $W$ distinct $\nu$ give distinct constant terms. Since $\C[s]$ is a unique factorization domain, $P_{D,i}$ divides $b_{D,i}$.

\subsubsection*{The upper bound}
Let $\ell$ be an irreducible factor of $b_{D,i}$ and put $\mathfrak q=(\ell)$. We show that $\ell$ is a constant multiple of one of the linear forms $L_{W,\nu}$. Let
\[
T:=\{h\in\{1,\dots,r\}\mid\tau_{e_h}(\mathfrak q)\neq\mathfrak q\},
\]
where $\tau_{e_h}(b)(s)=b(s-e_h)$ as in Lemma \ref{lem:ACC}; thus $h\notin T$ if and only if $\ell(s-e_h)$ is a constant multiple of $\ell(s)$, which is the case, for instance, when $\ell$ does not involve $s_h$. Put $\sN:=(M^{-e_i}_f)_{\mathfrak q}$. Since $\ell$ divides $b_{D,i}$ we have $\mathfrak q\supseteq B^{-e_i}_f$, so $\sN\neq 0$ by Lemma \ref{lem:Z=supp}. For $h\notin T$, Corollary \ref{cor:invertible} shows that $f_h$ acts invertibly on $(\sD_X[s]f^s)_{\mathfrak q}$ and on $(\sD_X[s]f^{s+e_i})_{\mathfrak q}$, hence on their quotient $\sN$. Let $g=\prod_{h\notin T}f_h$. Then $\sN=\sO_X[1/g]\otimes_{\sO_X}\sN$, so the restriction of $\sN$ to $X\setminus(g=0)$ is nonzero, and there is a point $y\in X$ with $g(y)\neq 0$ and $\sN_y=\big((M^{-e_i}_f)_y\big)_{\mathfrak q}\neq 0$ (faithful flatness of the analytic local ring). Let $W'\in L(D)$ be the edge with $y\in W'^\circ$, i.e. $J(W',f)=J_y=\{j\mid f_j(y)=0\}$; since $g(y)\neq 0$ we have $J(W',f)\subseteq T$.

The stalk $(M^{-e_i}_f)_y$ is relative holonomic, so by Lemma \ref{lem:Z=supp} the condition $\sN_y\neq 0$ means $\mathfrak q\supseteq B^{-e_i}_{f,y}$. By Lemma \ref{lem:localization-edge} this forces $i\in J(W',f)$ and $\ell\mid b_{D^{W'},i}$, where $b_{D^{W'},i}\in\C[s_{J(W',f)}]$. Factoring $b_{D^{W'},i}$ into irreducibles in $\C[s_{J(W',f)}]$, which remain irreducible in $\C[s]$, we see that $\ell$ is a constant multiple of a polynomial in the variables $s_j$, $j\in J(W',f)$; hence $\tau_{e_h}(\mathfrak q)=\mathfrak q$ for $h\notin J(W',f)$, that is, $T\subseteq J(W',f)$. Therefore
\begin{equation}\label{eq:T=J}
T=J(W',f).
\end{equation}

Next we show that $W'$ is dense. Let $D^{W'}=D^{(1)}\times\dots\times D^{(c)}$ be the decomposition of the essential arrangement $D^{W'}$ in $X/W'=V_1\oplus\dots\oplus V_c$ into irreducible factors, let $I_a\subseteq J(W',f)$ be the indices of the hyperplanes in $D^{(a)}$, and let $a_0$ be such that $i\in I_{a_0}$. Write $f_{(a)}=(f_j)_{j\in I_a}$, a complete factorization of $D^{(a)}$ in $V_a$. Since $f_{W'}^{s_{J(W',f)}}=\prod_af_{(a)}^{s_{I_a}}$ and the factors depend on disjoint sets of variables, the $\sD_{X/W'}[s_{J(W',f)}]$-modules generated by $f_{W'}^{s}$ and by $f_{W'}^{s+e_i}$ are the tensor products over $\C$ of the modules generated by the $f_{(a)}^{s_{I_a}}$, respectively by $f_{(a_0)}^{s_{I_{a_0}}+e_i}$ and $f_{(a)}^{s_{I_a}}$ ($a\neq a_0$). Hence
\[
M^{-e_i}_{f_{W'}}\simeq M^{-e_i}_{f_{(a_0)}}\otimes_\C\bigotimes_{a\neq a_0}\sD_{V_a}[s_{I_a}]f_{(a)}^{s_{I_a}} ,
\]
and the module $M_2:=\bigotimes_{a\neq a_0}\sD_{V_a}[s_{I_a}]f_{(a)}^{s_{I_a}}$ is torsion-free over $\C[s_{I'}]$, $I'=J(W',f)\setminus I_{a_0}$, being a submodule of the tensor product of the free $\C[s_{I_a}]$-modules $j_*(\sO_{U_a}[s_{I_a}]f_{(a)}^{s_{I_a}})$, which is free over $\C[s_{I'}]$. A polynomial $p=\sum_\gamma p_\gamma(s_{I_{a_0}})\,s_{I'}^\gamma$ annihilates a tensor product $M_1\otimes_\C M_2$ with $M_2$ torsion-free over $\C[s_{I'}]$ if and only if every $p_\gamma$ annihilates $M_1$: indeed, $p$ then annihilates $M_1\otimes_\C(M_2\otimes_{\C[s_{I'}]}\Frac(\C[s_{I'}]))$, which is a direct sum of copies of $M_1\otimes_\C\Frac(\C[s_{I'}])$, and the monomials $s_{I'}^\gamma$ are linearly independent over $\C$ in $\Frac(\C[s_{I'}])$. Consequently $B^{-e_i}_{f_{W'}}=B^{-e_i}_{f_{(a_0)}}\cdot\C[s_{J(W',f)}]$ and $b_{D^{W'},i}\in\C[s_{I_{a_0}}]$ up to a constant. As above, $\ell$ is then a constant multiple of a polynomial in the variables $s_{I_{a_0}}$, so $T\subseteq I_{a_0}$, and \eqref{eq:T=J} gives $I_{a_0}=J(W',f)$, i.e. $c=1$ and $D^{W'}$ is irreducible. Thus $W'$ is a dense edge with $i\in J(W',f)$.

Finally we apply Section \ref{sec:local} to the essential, irreducible, free arrangement $D^{W'}$ in $X/W'\simeq\C^{n'}$, $n'=\rank(W')$, with its $r'=|J(W',f)|$ hyperplanes, the parameters $s_{J(W',f)}$, and the index $i$ in place of $1$. Let $\mathfrak q'=\ell\,\C[s_{J(W',f)}]$, a prime ideal of $\C[s_{J(W',f)}]$ containing $B^{-e_i}_{f_{W'}}$. For every proper subset $J\subsetneq J(W',f)$ we have $\mathfrak q'\cap\C[s_J]=0$: otherwise $\ell$ would divide a nonzero polynomial in the variables $s_J$ and, as above, would be a constant multiple of a polynomial in these variables, contradicting \eqref{eq:T=J}. Hence Lemma \ref{lem:support-origin} applies to $\mathfrak q'$, and $\sN':=(M^{-e_i}_{f_{W'}})_{\mathfrak q'}$ is a nonzero module (Lemma \ref{lem:Z=supp}), supported at the origin of $X/W'$, and annihilated by $\ell$ because $B^{-e_i}_{f_{W'}}$ is principal and radical. By Lemma \ref{lem:kashiwara}, $\sN'$ is a nonzero direct sum of copies of $\delta_{0,\K'}$, $\K'=\Frac(\C[s_{J(W',f)}]/\mathfrak q')$, and by the cyclic presentation \eqref{eq:cyclic-presentation-unit} there is a nonzero $v\in\delta_{0,\K'}$ with $f_iv=0$ and $\wt Ev=0$, $\wt E=E-\one_{J(W',f)}\cdot s$. As $E$ acts on $\partial^\beta\delta_0$ by $-(n'+|\beta|)$, this forces $\one_{J(W',f)}\cdot s=-(n'+q)$ in $\K'$ for some $q\in\N$, that is, $\one_{J(W',f)}\cdot s+n'+q\in\mathfrak q'=(\ell)$. Since $\ell$ is irreducible, it is a constant multiple of the linear form $\one_{J(W',f)}\cdot s+n'+q=L_{W',q}$, and Theorem \ref{thm:local} applied to $D^{W'}$ shows that $0\leq q\leq r'-2n'+1$. Thus $\ell$ is a constant multiple of one of the factors of $P_{D,i}$.

Since $b_{D,i}$ is squarefree and each of its irreducible factors divides $P_{D,i}$, we conclude that $b_{D,i}$ divides $P_{D,i}$. Together with the lower bound this proves $B^{-e_i}_f=(P_{D,i}(s))$, i.e. Theorem \ref{thm:main-unit}. \qed

\begin{proof}[Proof of Corollary \ref{cor:general-shift}]
By translation (\S\ref{sec:relD}), $M^{a,b}_f$ is $M^{-c}_f=M^{0,-c}_f$ with the parameters shifted by $a$, so $B^{a,b}_f=\tau_a(B^{-c}_f)=\{p(s-a)\mid p\in B^{-c}_f\}$, and it suffices to prove $B^{-c}_f=(P_{D,c}(s))$. Put $N=\sum_jc_j$ and choose a chain
\[
0=a^{(0)}\geq a^{(1)}\geq\dots\geq a^{(N)}=-c\qquad\text{with}\qquad a^{(m+1)}=a^{(m)}-e_{c_m},\quad c_m\in\{1,\dots,r\}.
\]
For $m<m'$ the ideal $B^{a^{(m)},a^{(m')}}_f$ is principal and radical by Theorem \ref{thm:Bath-principal-radical}, and it is nonzero because $Z(B^{a^{(m)},a^{(m')}}_f)=p_2(\Chrel(M^{a^{(m)},a^{(m')}}_f))$ has dimension at most $r-1$ by Theorem \ref{thm:Wu-basic}(2); hence its zero locus is empty or of pure dimension $r-1$. Theorem \ref{thm:Wu-basic}(3) and induction on $N$ therefore give
\begin{equation}\label{eq:zero-locus-chain}
Z(B^{-c}_f)=\bigcup_{m=0}^{N-1}Z\big(B^{a^{(m)},a^{(m+1)}}_f\big)=\bigcup_{m=0}^{N-1}\Big(Z(B^{-e_{c_m}}_f)+a^{(m)}\Big),
\end{equation}
the second equality by translation. By Theorem \ref{thm:main-unit}, $Z(B^{-e_{c_m}}_f)$ is the union of the hyperplanes $(L_{W,\nu}=0)$ over the dense edges $W$ with $c_m\in J(W,f)$ and $0\leq\nu\leq\rho_W:=|J(W,f)|-2\rank(W)+1$, and the translate of $(L_{W,\nu}=0)$ by $a^{(m)}$ is the hyperplane $\big(\sum_{j\in J(W,f)}s_j+\rank(W)+\nu-\one_{J(W,f)}\cdot a^{(m)}=0\big)$. Fix a dense edge $W$. The steps $m$ with $c_m\in J(W,f)$ are exactly those at which the sum $\one_{J(W,f)}\cdot a^{(m)}$ decreases; this sum starts at $0$ and ends at $-c_W$, so at these steps it takes the values $0,-1,\dots,-(c_W-1)$, each exactly once. Hence $W$ contributes to \eqref{eq:zero-locus-chain} the hyperplanes $\big(\sum_{j\in J(W,f)}s_j+\rank(W)+\nu+t=0\big)$ with $0\leq\nu\leq\rho_W$ and $0\leq t\leq c_W-1$, and nothing if $c_W=0$. Since $\rho_W\geq 0$, the integers $\nu+t$ fill the interval $[0,\rho_W+c_W-1]=[0,|J(W,f)|-2\rank(W)+c_W]$. Therefore $Z(B^{-c}_f)$ is the union of the pairwise distinct hyperplanes defined by the linear factors of $P_{D,c}$. As $B^{-c}_f$ is principal and radical, its generator is a squarefree polynomial with this zero locus, hence a nonzero constant multiple of $P_{D,c}$.
\end{proof}


\begin{remark}\label{rem:symmetry}
Bath's duality formula \cite[Theorem 1.2]{Bat20b}, applied with $f'=1$ and $g=f_i$, gives $\mathbb D(M^{-e_i}_f)\simeq\sD_X[s]f^{-s-\one_r-e_i}/\sD_X[s]f^{-s-\one_r}[n+1]$ for a free arrangement; since $M^{-e_i}_f$ is $(n+1)$-Cohen-Macaulay (\cite[Lemma 2.5 and Theorem 5.2]{Wu20}) this implies that $b(s)\in B^{-e_i}_f$ if and only if $b(-s-\one_r-e_i)\in B^{-e_i}_f$. The polynomial $P_{D,i}$ visibly has this symmetry: the involution $s\mapsto-s-\one_r-e_i$ maps $L_{W,\nu}$ to $-L_{W,\nu'}$ with $\nu'=|J(W,f)|-2\rank(W)+1-\nu$, for every dense edge $W$ with $i\in J(W,f)$.
\end{remark}

\begin{corollary}\label{cor:unit-zero-locus}
Let $D$ be a free central arrangement with complete factorization $f$. Then $Z(B^{-e_i}_f)$ is the union of the hyperplanes $(L_{W,\nu}=0)$ over the dense edges $W$ with $i\in J(W,f)$ and $0\leq\nu\leq|J(W,f)|-2\rank(W)+1$, and its irreducible components are exactly the projections of the components $T^*_WX\times(L_{W,\nu}=0)$ of $\Chrel(M^{-e_i}_f)$.
\end{corollary}

\begin{proof}
The first assertion is Theorem \ref{thm:main-unit}. For the last one, the proof of the lower bound shows that $(L_{W,\nu}=0)$ is the projection of the component $T^*_{\{0\}}(X/W)\times(L_{W,\nu}=0)$ of $\Chrel(M^{-e_i}_{f_W})$, whose pullback to a neighborhood of a general point of $W$ is $T^*_WX\times(L_{W,\nu}=0)$; since the characteristic variety is local on $X$, this is a component of $\Chrel(M^{-e_i}_f)$.
\end{proof}

\section{Coordinate monoid ideals}\label{sec:coord}

In this section $D$ is a central arrangement in $X=\C^n$ with complete factorization $f=(f_1,\dots,f_r)$; freeness is assumed only where stated. Let $W_0=\bigcap_{j=1}^rD_j$ be the center of $D$ and $\ell=\rank(W_0)=\codim_XW_0$. Since $W_0$ is the smallest edge, $D_{W_0}=D$ and $D^{W_0}$ is an essential arrangement in $X/W_0$. Choose a linear complement $V$ of $W_0$ in $X$, so that $X=V\oplus W_0$ and $V\simeq X/W_0$, and let \eqref{eq:product-decomp} be the decomposition of $D^{W_0}$ into essential irreducible factors $D^{(a)}$ in $V_a$, $V=V_1\oplus\dots\oplus V_c$; recall that $I_a$ denotes the set of indices $j$ with $D_j/W_0\in D^{(a)}$, $n_a=\dim V_a$, and $\lambda^{(a)}=\sum_{i\in I_a}s_i+n_a$. Choose linear coordinates $y_1,\dots,y_\ell$ on $V$ adapted to the decomposition $V=\bigoplus_aV_a$, and coordinates $z_1,\dots,z_{n-\ell}$ on $W_0$. Every $f_j$ is a linear form on $V$, i.e. depends only on the $y$-coordinates, and $D$ is the product of $D^{W_0}$ with the empty arrangement in $W_0$.

\begin{lemma}\label{lem:coordinate-ideal}
$(f_1,\dots,f_r)=(y_1,\dots,y_\ell)$ as ideals of $\C[X]$.
\end{lemma}

\begin{proof}
The linear forms vanishing on $W_0$ form the subspace $W_0^\perp\subseteq X^*$, of dimension $\ell$, with basis $y_1,\dots,y_\ell$. The $f_j$ lie in $W_0^\perp$, and their common kernel is $W_0$; hence they span $W_0^\perp$, and the ideals generated by two spanning sets of the same subspace of linear forms coincide.
\end{proof}

\subsection{The full coordinate quotient}\label{sec:full-coordinate}
Let
\begin{equation}\label{eq:L}
\sL:=\sD_X\big(y_1,\dots,y_\ell,\partial_{z_1},\dots,\partial_{z_{n-\ell}}\big)\subseteq\sD_X ,
\end{equation}
so that $\sD_X/\sL$ is the pullback to $X$ of the $\delta$-module at the origin of $X/W_0$.

\begin{theorem}\label{thm:full-coordinate}
Let $D$ be a central arrangement with complete factorization $f$, and let $\Lambda=(\lambda^{(1)},\dots,\lambda^{(c)})\subseteq\C[s]$.
\begin{enumerate}
\item $\Lambda\subseteq B^{K_{[r]}}_f$; that is, $\lambda^{(a)}f^s\in\sum_{j=1}^r\sD_X[s]f^{s+e_j}$ for $a=1,\dots,c$.
\item If $D$ is free, then there is an isomorphism of $\sD_X[s]$-modules
\begin{equation}\label{eq:N0-structure}
N^{K_{[r]}}_f\simeq(\sD_X/\sL)\otimes_\C\C[s]/\Lambda ,
\end{equation}
and consequently
\begin{equation}\label{eq:full-coordinate-formula}
B^{K_{[r]}}_f=\big(\lambda^{(1)},\dots,\lambda^{(c)}\big)=\Big(\sum_{i\in I_1}s_i+n_1,\ \dots,\ \sum_{i\in I_c}s_i+n_c\Big).
\end{equation}
\end{enumerate}
\end{theorem}

\begin{proof}
Let $E_a=\sum_{y_j\in V_a}y_j\partial_{y_j}$ be the Euler vector field of $V_a$, viewed as a vector field on $X$. For $i\in I_a$ the linear form $f_i$ depends only on the coordinates of $V_a$, so $E_a(f_i)=f_i$, while $E_a(f_i)=0$ for $i\notin I_a$. Hence $E_a$ is logarithmic and
\begin{equation}\label{eq:Euler-tilde}
\wt E_a=E_a-\sum_{i\in I_a}s_i\qquad\text{annihilates } f^s .
\end{equation}
By the Weyl commutation relations, $E_a+n_a=\sum_{y_j\in V_a}\partial_{y_j}y_j$. Therefore
\[
\lambda^{(a)}f^s=(E_a+n_a)f^s-\wt E_af^s=\sum_{y_j\in V_a}\partial_{y_j}\big(y_jf^s\big),
\]
and $y_jf^s\in\sum_i\sD_X[s]f^{s+e_i}$ because $y_j\in(f_1,\dots,f_r)$ by Lemma \ref{lem:coordinate-ideal} and $f_if^s=f^{s+e_i}$. This proves (1).

Now assume that $D$ is free. By Lemma \ref{lem:exponent-one}(3) applied to $D^{W_0}$, together with Theorem \ref{thm:free-facts}(2) for the product of $D^{W_0}$ with the empty arrangement in $W_0$, the vector fields
\[
E_a,\ \theta_{a,2},\dots,\theta_{a,n_a}\ (a=1,\dots,c),\qquad \partial_{z_1},\dots,\partial_{z_{n-\ell}}
\]
form a homogeneous basis of $\Der(-\log D)$, where each $\theta_{a,j}$ has coefficient degree at least $2$ and involves only the coordinates of $V_a$. By Theorem \ref{thm:log-ann}, Lemma \ref{lem:cyclic} and Lemma \ref{lem:coordinate-ideal},
\begin{equation}\label{eq:N0-presentation}
N^{K_{[r]}}_f\simeq\frac{\sD_X[s]}{\sD_X[s]\big(\wt E_a,\wt\theta_{a,j},\wt\partial_{z_q}\big)+\sD_X[s](y_1,\dots,y_\ell)} .
\end{equation}
We compare the left ideal in the denominator with $\sL[s]+\sD_X[s]\Lambda$, where $\sL[s]=\sL\otimes_\C\C[s]$. First, every $f_i$ is independent of the $z$-coordinates, so $\wt\partial_{z_q}=\partial_{z_q}\in\sL$. Second, let $\theta=\sum_ja_j(y)\partial_{y_j}$ be one of the $\theta_{a,j}$, of coefficient degree $d\geq 2$, and put $b_i=\theta(f_i)/f_i$, a homogeneous polynomial of degree $d-1\geq 1$. Then $a_j$, $\partial a_j/\partial y_j$ and $b_i$ all lie in the ideal $(y_1,\dots,y_\ell)$, and the Weyl relation $a_j\partial_{y_j}=\partial_{y_j}a_j-\partial a_j/\partial y_j$ shows that
\begin{equation}\label{eq:theta-in-L}
\wt\theta=\theta-\sum_is_ib_i\in\sD_X[s](y_1,\dots,y_\ell)\subseteq\sL[s].
\end{equation}
Third, by \eqref{eq:Euler-tilde} and the relation $E_a+n_a=\sum_{y_j\in V_a}\partial_{y_j}y_j\in\sL$,
\begin{equation}\label{eq:Euler-in-L}
\wt E_a+\lambda^{(a)}=E_a+n_a\in\sL[s].
\end{equation}
These three facts give the inclusion ``$\subseteq$'' in
\begin{equation}\label{eq:ideal-equality}
\sD_X[s]\big(\wt E_a,\wt\theta_{a,j},\wt\partial_{z_q}\big)+\sD_X[s](y_1,\dots,y_\ell)=\sL[s]+\sD_X[s]\Lambda .
\end{equation}
Conversely, the generators $y_j$ of $\sL$ are present on the left, the $\partial_{z_q}$ are among the logarithmic generators, and $\lambda^{(a)}=(E_a+n_a)-\wt E_a$ lies in the left-hand side by \eqref{eq:Euler-in-L}, since $E_a+n_a\in\sD_X(y_1,\dots,y_\ell)$. This proves \eqref{eq:ideal-equality}. Since the $\lambda^{(a)}$ are central, \eqref{eq:N0-presentation} and \eqref{eq:ideal-equality} give
\[
N^{K_{[r]}}_f\simeq\frac{\sD_X\otimes_\C\C[s]}{\sL\otimes_\C\C[s]+\sD_X\otimes_\C\Lambda}\simeq(\sD_X/\sL)\otimes_\C\C[s]/\Lambda ,
\]
which is \eqref{eq:N0-structure}. Finally, $\sD_X/\sL$ is a nonzero $\C$-vector space, so $(\sD_X/\sL)\otimes_\C\C[s]/\Lambda$ is a nonzero direct sum of copies of $\C[s]/\Lambda$ as a $\C[s]$-module; a polynomial annihilates it if and only if its image in $\C[s]/\Lambda$ is zero. Hence $B^{K_{[r]}}_f=\Lambda$.
\end{proof}

\begin{corollary}\label{cor:essential-irreducible}
If $D$ is a free, essential and irreducible central arrangement in $\C^n$ with complete factorization $f$, then
\[
N^{K_{[r]}}_f\simeq\Delta\otimes_\C\C[s]\Big/\Big(\sum_{i=1}^rs_i+n\Big)\qquad\text{and}\qquad B^{K_{[r]}}_f=\Big(\sum_{i=1}^rs_i+n\Big).
\]
\end{corollary}

\begin{proof}
Here $W_0=\{0\}$, $c=1$, $I_1=\{1,\dots,r\}$, $n_1=n$ and $\sD_X/\sL=\Delta$.
\end{proof}

Theorem \ref{thm:main-coord} is Theorem \ref{thm:full-coordinate}(2) and Corollary \ref{cor:essential-irreducible}. Note that part (1) of Theorem \ref{thm:full-coordinate} holds for arbitrary central arrangements and contains \cite[Lemma 5.5]{Wu20}; freeness is used only to show that no further parameter relation occurs, via the logarithmic annihilator and the Poincar\'e-Birkhoff-Witt basis of $\sD_X/\sL$.

\section{A nonlinear component}\label{sec:counter}

In this section $X=\C^2$ with coordinates $x,y$, $\sD_X=\C\langle x,y,\partial_x,\partial_y\rangle$ is the Weyl algebra,
\begin{equation}\label{eq:four-lines}
f=(f_1,f_2,f_3,f_4)=(x,\ y,\ x+y,\ x+2y),\qquad f_D=xy(x+y)(x+2y),
\end{equation}
$\C[s]=\C[s_1,s_2,s_3,s_4]$, and $K=\mon{3e_1,3e_2}=(3e_1+\N^4)\cup(3e_2+\N^4)$. We write
\begin{equation}\label{eq:lambda-d}
\lambda_d:=s_1+s_2+s_3+s_4+2+d\qquad(d\in\Z),
\end{equation}
which agrees with \eqref{eq:lambda-nu} for $n=2$, and we let $\Phi$ be the polynomial \eqref{eq:Phi-intro}, i.e.
\begin{equation}\label{eq:Phi}
\begin{aligned}
\Phi&=(2s_1-s_2+s_3+1)(2s_1-s_2+s_3+4)+2(s_1+2)(s_2+2)\\
&=4s_1^2-2s_1s_2+4s_1s_3+s_2^2-2s_2s_3+s_3^2+14s_1-s_2+5s_3+12 .
\end{aligned}
\end{equation}

\begin{theorem}\label{thm:counter}
For $f$ and $K$ as above,
\begin{equation}\label{eq:counter-formula}
B^K_f=(\lambda_0)\cap(\lambda_1)\cap(\lambda_2)\cap(\lambda_3,\Phi)\cap(s_1+3,\,s_2+3,\,s_3,\,s_4).
\end{equation}
The five ideals on the right-hand side are prime, pairwise comaximal, and $B^K_f$ is radical. The ideal $(\lambda_3,\Phi)$ defines an irreducible nonlinear quadric surface, so $Z(B^K_f)$ is not a finite union of translated linear subvarieties.
\end{theorem}

Theorem \ref{thm:counter} is Theorem \ref{thm:main-counter}; the last assertion disproves Conjecture E for the coefficient module $\sO_X$, even for a reduced free central line arrangement. The purpose of the proof below is to make the computation entirely finite-dimensional and explicit.

\subsection{The logarithmic cyclic presentation}\label{sec:log-cyclic}
The derivations
\begin{equation}\label{eq:E-delta}
E=x\partial_x+y\partial_y,\qquad\delta=f_{D,y}\,\partial_x-f_{D,x}\,\partial_y,\qquad f_{D,x}=\partial_xf_D,\ f_{D,y}=\partial_yf_D ,
\end{equation}
satisfy $\det\begin{pmatrix}x&y\\ f_{D,y}&-f_{D,x}\end{pmatrix}=-xf_{D,x}-yf_{D,y}=-4f_D$. Since $\delta(f_D)=0$ and $f_D$ is reduced, $\delta$ is logarithmic, and Saito's criterion (Theorem \ref{thm:free-facts}(1)) shows that $E,\delta$ is a basis of $\Der(-\log D)$, with exponents $(1,3)$; thus $D$ is free, essential and irreducible, and $\rho=r-2n+1=1$ in the notation of \eqref{eq:rho}. Let $\varrho_i=\delta(f_i)/f_i$. Using $f_{D,y}=x^3+6x^2y+6xy^2$ and $f_{D,x}=3x^2y+6xy^2+2y^3$, direct differentiation gives
\begin{equation}\label{eq:rho-i}
\varrho_1=x^2+6xy+6y^2,\quad\varrho_2=-3x^2-6xy-2y^2,\quad\varrho_3=x^2+2xy-2y^2,\quad\varrho_4=x^2-2xy-2y^2 .
\end{equation}
The logarithmic operators annihilating $f^s$ are
\begin{equation}\label{eq:PE-Pdelta}
P_E:=\wt E=E-(s_1+s_2+s_3+s_4),\qquad P_\delta:=\wt\delta=\delta-\sum_{i=1}^4s_i\varrho_i ,
\end{equation}
and Theorem \ref{thm:log-ann} gives $\Ann_{\sD_X[s]}f^s=\sD_X[s](P_E,P_\delta)$. Since $f^{s+3e_1}=x^3f^s$ and $f^{s+3e_2}=y^3f^s$, Lemma \ref{lem:cyclic} yields
\begin{equation}\label{eq:NK-presentation}
N^K_f\simeq\frac{\sD_X[s]}{\sD_X[s](P_E,P_\delta,x^3,y^3)},\qquad B^K_f=\Ann_{\C[s]}N^K_f .
\end{equation}

\subsection{The jet model}\label{sec:jet-model}
Let $\Delta=\sD_X/\sD_X(x,y)$ be the $\delta$-module at the origin with cyclic vector $\delta_0$, and put $e_{ij}=\partial_x^i\partial_y^j\delta_0$ for $i,j\geq 0$. By \eqref{eq:delta-rules},
\begin{equation}\label{eq:jet-rules}
x\,e_{ij}=-i\,e_{i-1,j},\qquad y\,e_{ij}=-j\,e_{i,j-1},\qquad\partial_x\,e_{ij}=e_{i+1,j},\qquad\partial_y\,e_{ij}=e_{i,j+1}.
\end{equation}
Representing $e_{ij}$ by the monomial $u^iv^j$, so that $\Delta\simeq\C[u,v]$, a normally ordered monomial of the Weyl algebra acts by
\begin{equation}\label{eq:weyl-action}
x^ay^b\partial_x^c\partial_y^e\cdot w=(-\partial_u)^a(-\partial_v)^b\big(u^cv^ew\big)\qquad(w\in\C[u,v]).
\end{equation}
For $w=\sum a_{ij}e_{ij}$ the equation $x^3w=0$ is equivalent to $a_{ij}=0$ for $i\geq 3$, and $y^3w=0$ to $a_{ij}=0$ for $j\geq 3$. Thus the relevant subspace of $\Delta$ is the nine-dimensional \emph{jet space}
\begin{equation}\label{eq:J}
J:=\Span_\C\{u^iv^j\mid 0\leq i,j\leq 2\}=\bigoplus_{d=0}^4J_d,\qquad n_d:=\dim_\C J_d ,
\end{equation}
where $J_d$ is spanned by the monomials of total degree $d$, with the ordered bases
\begin{equation}\label{eq:jet-bases}
J_0=(1),\quad J_1=(v,u),\quad J_2=(v^2,uv,u^2),\quad J_3=(uv^2,u^2v),\quad J_4=(u^2v^2).
\end{equation}

\begin{lemma}\label{lem:jet-operators}
The operators $P_E$ and $P_\delta$ act on $J\otimes_\C\C[s]$ by the $\C[s]$-linear maps
\begin{align}
L_E&=-\big(u\partial_u+v\partial_v+\lambda_0\big),\qquad\text{so that } L_E|_{J_d}=-\lambda_d\cdot\mathrm{id},\label{eq:LE}\\
L_\delta&=(3v-6u)\partial_u^2\partial_v+(6v-6u)\partial_u\partial_v^2-\alpha_0\partial_u^2-\beta_0\partial_u\partial_v-\gamma_0\partial_v^2,\qquad L_\delta(J_d)\subseteq J_{d-2},\label{eq:Ldelta}
\end{align}
where
\begin{equation}\label{eq:abc}
\alpha_0=s_1-3s_2+s_3+s_4,\qquad\beta_0=6s_1-6s_2+2s_3-2s_4,\qquad\gamma_0=6s_1-2s_2-2s_3-2s_4 .
\end{equation}
\end{lemma}

\begin{proof}
By \eqref{eq:weyl-action}, $x\partial_x$ and $y\partial_y$ act by $-\partial_uu=-(1+u\partial_u)$ and $-(1+v\partial_v)$, which gives \eqref{eq:LE}. For $P_\delta$, the quadratic forms $\varrho_i$ act by the corresponding second order operators $\partial_u^2,\partial_u\partial_v,\partial_v^2$ (the two signs cancel), and $\alpha_0,\beta_0,\gamma_0$ are the coefficients of $x^2$, $xy$, $y^2$ in $\sum_is_i\varrho_i$ by \eqref{eq:rho-i}. The operator $\delta=f_{D,y}\partial_x-f_{D,x}\partial_y$ acts by
\[
-\partial_u^3(u\,\cdot)-6\partial_u^2\partial_v(u\,\cdot)-6\partial_u\partial_v^2(u\,\cdot)+3\partial_u^2\partial_v(v\,\cdot)+6\partial_u\partial_v^2(v\,\cdot)+2\partial_v^3(v\,\cdot),
\]
and expanding with the product rule, the lower order terms $-3\partial_u^2+3\partial_u^2$, $-12\partial_u\partial_v+12\partial_u\partial_v$ and $-6\partial_v^2+6\partial_v^2$ cancel, while the terms $-u\partial_u^3$ and $2v\partial_v^3$ vanish on $J$. This gives \eqref{eq:Ldelta}; the formula visibly lowers the total degree by two, and one checks on the monomials of $J$ that it maps $J_d$ into $J_{d-2}$ (with $J_{-1}=J_{-2}=0$).
\end{proof}

In the bases \eqref{eq:jet-bases}, the nonzero blocks of $L_\delta$ before imposing any Euler relation are
\begin{equation}\label{eq:tilde-M}
\wt M_2=\begin{pmatrix}-2\gamma_0&-\beta_0&-2\alpha_0\end{pmatrix},\qquad
\wt M_3=\begin{pmatrix}12-2\beta_0&6-2\alpha_0\\-12-2\gamma_0&-12-2\beta_0\end{pmatrix},\qquad
\wt M_4=\begin{pmatrix}12-2\alpha_0\\-4\beta_0\\-24-2\gamma_0\end{pmatrix},
\end{equation}
where $\wt M_d$ is the $n_{d-2}\times n_d$ matrix of $L_\delta|_{J_d}\colon J_d\to J_{d-2}$; the blocks for $d=0,1$ are zero. Thus, with respect to the basis $(1,v,u,v^2,uv,u^2,uv^2,u^2v,u^2v^2)$ of $J$,
\begin{equation}\label{eq:LE-matrix}
[L_E]=\operatorname{diag}\big(-\lambda_0,-\lambda_1,-\lambda_1,-\lambda_2,-\lambda_2,-\lambda_2,-\lambda_3,-\lambda_3,-\lambda_4\big),
\end{equation}
and $[L_\delta]$ is the $6\times 9$ matrix, with rows indexed by the basis $(1,v,u,v^2,uv,u^2)$ of $J_{\leq 2}=J_0\oplus J_1\oplus J_2$, whose only nonzero entries are the blocks $\wt M_2$ (row $1$, columns $4$--$6$), $\wt M_3$ (rows $2$--$3$, columns $7$--$8$) and $\wt M_4$ (rows $4$--$6$, column $9$).

\subsection{The universal jet presentation}\label{sec:universal-jet}
Every element of $N^K_f$ is annihilated by a power of $(x,y)$, because $x^3$ and $y^3$ annihilate the cyclic generator and $\sD_X[s]$ is generated by $\partial_x,\partial_y$ over $\C[x,y,s]$. By Lemma \ref{lem:kashiwara},
\begin{equation}\label{eq:NK-Kashiwara}
N^K_f\simeq\Delta\otimes_\C\sV,\qquad\sV:=\{w\in N^K_f\mid xw=yw=0\},\qquad B^K_f=\Ann_{\C[s]}\sV .
\end{equation}
The following proposition upgrades the fiberwise rank computation on the jet space to a presentation of the $\C[s]$-module $\sV$ itself; this is the conceptual bridge from the $\sD_X[s]$-module $N^K_f$ to finite-dimensional linear algebra.

\begin{proposition}\label{prop:universal-jet}
Let $J_{\C[s]}=J\otimes_\C\C[s]$ and $(J_{\leq 2})_{\C[s]}=J_{\leq 2}\otimes_\C\C[s]$, and consider the $\C[s]$-linear map
\[
A\colon J_{\C[s]}\to J_{\C[s]}\oplus(J_{\leq 2})_{\C[s]},\qquad A(z)=(L_Ez,L_\delta z),
\]
whose matrix in the bases \eqref{eq:jet-bases} is the $15\times 9$ matrix $[A]=\binom{[L_E]}{[L_\delta]}$. Then
\begin{equation}\label{eq:V-coker}
\sV\simeq\coker\big(A^t\colon\C[s]^{15}\to\C[s]^{9}\big).
\end{equation}
Moreover, if $M_d$ denotes the reduction of $\wt M_d$ modulo $\lambda_d$ (obtained by the substitution $s_4=-d-2-s_1-s_2-s_3$), then
\begin{equation}\label{eq:V-blocks}
\sV\simeq\bigoplus_{d=0}^4\sV_d,\qquad\sV_d=\frac{\C[s]^{n_d}}{\lambda_d\,\C[s]^{n_d}+\im\big(M_d^t\big)} .
\end{equation}
\end{proposition}

\begin{proof}
We compare two functors on the category of $\C[s]$-modules. Let $W$ be an arbitrary $\C[s]$-module; no flatness, torsion-freeness or finite generation assumption on $W$ is needed.

\emph{Step 1.} By \eqref{eq:NK-Kashiwara} and Lemma \ref{lem:kashiwara}(2), the map $\phi\mapsto\mathrm{id}_\Delta\otimes\phi$ is a natural isomorphism
\[
\Hom_{\C[s]}(\sV,W)\xrightarrow{\ \sim\ }\Hom_{\sD_X[s]}(N^K_f,\Delta\otimes_\C W).
\]

\emph{Step 2.} By \eqref{eq:NK-presentation}, a $\sD_X[s]$-linear map $\varphi\colon N^K_f\to\Delta\otimes_\C W$ is determined by $w=\varphi(1)$, and $w$ ranges exactly over the elements of $\Delta\otimes_\C W$ satisfying
\begin{equation}\label{eq:four-relations}
x^3w=0,\qquad y^3w=0,\qquad P_Ew=0,\qquad P_\delta w=0 .
\end{equation}

\emph{Step 3.} Write $w=\sum_{i,j}u^iv^j\otimes w_{ij}$ with $w_{ij}\in W$. By \eqref{eq:jet-rules}, $x^3w=0$ if and only if $w_{ij}=0$ for $i\geq 3$, and $y^3w=0$ if and only if $w_{ij}=0$ for $j\geq 3$; hence the first two conditions in \eqref{eq:four-relations} say precisely that $w\in J\otimes_\C W=J_{\C[s]}\otimes_{\C[s]}W$. For such $w$, the operators $P_E$ and $P_\delta$ act as $L_E\otimes_{\C[s]}\mathrm{id}_W$ and $L_\delta\otimes_{\C[s]}\mathrm{id}_W$, so the last two conditions say that $(A\otimes_{\C[s]}W)(w)=0$. Combining Steps 1--3 we obtain a natural isomorphism
\begin{equation}\label{eq:Hom-is-ker}
\Hom_{\C[s]}(\sV,W)\simeq\ker\big(A\otimes_{\C[s]}W\big).
\end{equation}
Note that we identify the $\Hom$-set with the kernel of the base-changed map; we do not use the (generally false) identity $\ker(A)\otimes W=\ker(A\otimes W)$.

\emph{Step 4.} Put $F=J_{\C[s]}$ and $G=J_{\C[s]}\oplus(J_{\leq 2})_{\C[s]}$, finite free $\C[s]$-modules, let $A^\vee\colon G^\vee\to F^\vee$ be the dual map, whose matrix in the dual bases is $[A]^t$, and let $C_A=\coker(A^\vee)$. Applying the left exact functor $\Hom_{\C[s]}(-,W)$ to the exact sequence $G^\vee\to F^\vee\to C_A\to 0$ gives an exact sequence $0\to\Hom(C_A,W)\to\Hom(F^\vee,W)\to\Hom(G^\vee,W)$. Since $F$ and $G$ are finite free, evaluation gives canonical isomorphisms $F\otimes_{\C[s]}W\simeq\Hom(F^\vee,W)$, $f\otimes w\mapsto(\eta\mapsto\eta(f)w)$, and $G\otimes_{\C[s]}W\simeq\Hom(G^\vee,W)$, under which the last arrow is $A\otimes_{\C[s]}\mathrm{id}_W$. Hence, naturally in $W$,
\begin{equation}\label{eq:ker-is-Hom}
\ker\big(A\otimes_{\C[s]}W\big)\simeq\Hom_{\C[s]}(C_A,W).
\end{equation}
In matrix terms: a map $\C[s]^9\to W$ is a column $w=(w_1,\dots,w_9)^t$ with entries in $W$, and it factors through $C_A=\coker([A]^t)$ if and only if it kills every column of $[A]^t$, i.e. if and only if $[A]w=0$.

\emph{Step 5.} By \eqref{eq:Hom-is-ker} and \eqref{eq:ker-is-Hom} the functors $\Hom_{\C[s]}(\sV,-)$ and $\Hom_{\C[s]}(C_A,-)$ are naturally isomorphic, so $\sV\simeq C_A=\coker(A^t)$ by Yoneda's lemma. This proves \eqref{eq:V-coker}.

\emph{Step 6.} The rows of $[A]$ are the rows $-\lambda_d\,e_k^t$ of $[L_E]$, one for each basis vector $e_k$ of $J_d$, and the rows of $[L_\delta]$; a row of $[L_\delta]$ indexed by a basis vector of $J_{d-2}$ has nonzero entries only in the columns of $J_d$, where it is the corresponding row of $\wt M_d$, by Lemma \ref{lem:jet-operators}. Hence $C_A=\C[s]^9/(\text{row space of }[A])$ splits as the direct sum over $d$ of $\C[s]^{n_d}/(\lambda_d\C[s]^{n_d}+\im\wt M_d^t)$. Finally, if $\wt M_d-M_d=\lambda_dN_d$, then $\lambda_d\C[s]^{n_d}+\im\wt M_d^t=\lambda_d\C[s]^{n_d}+\im M_d^t$, which proves \eqref{eq:V-blocks}.
\end{proof}

\subsection{The three small matrices}\label{sec:small-matrices}
On degree $d$, the Euler relation $\lambda_d=0$ permits the substitution $s_4=-d-2-s_1-s_2-s_3$ in \eqref{eq:tilde-M}. For $d=0,1$ the operator $L_\delta$ is zero. For $d=2$ we get
\begin{equation}\label{eq:M2}
M_2=\begin{pmatrix}-16(s_1+1)&-4(2s_1-s_2+s_3+2)&8(s_2+1)\end{pmatrix}.
\end{equation}
For $d=3$,
\begin{equation}\label{eq:M3}
M_3=\begin{pmatrix}-8(2s_1-s_2+s_3+1)&8(s_2+2)\\-16(s_1+2)&-8(2s_1-s_2+s_3+4)\end{pmatrix},\qquad\det M_3=64\,\Phi ,
\end{equation}
with $\Phi$ as in \eqref{eq:Phi}. For $d=4$,
\begin{equation}\label{eq:M4}
M_4=\begin{pmatrix}8(s_2+3)\\-16(2s_1-s_2+s_3+3)\\-16(s_1+3)\end{pmatrix}.
\end{equation}

\subsection{Annihilators of the five blocks}\label{sec:five-blocks}
We now apply Proposition \ref{prop:universal-jet}. Since the annihilator of a direct sum is the intersection of the annihilators, $B^K_f=\Ann_{\C[s]}\sV=\bigcap_{d=0}^4\Ann_{\C[s]}\sV_d$.

\emph{Degrees zero and one.} Since $M_0=M_1=0$, we have $\sV_0=\C[s]/(\lambda_0)$, $\sV_1=(\C[s]/(\lambda_1))^2$, and $\Ann\sV_0=(\lambda_0)$, $\Ann\sV_1=(\lambda_1)$.

\emph{Degree two.} Let $R=\C[s]/(\lambda_2)$, a polynomial ring in three variables. Then $\sV_2=R^3/R\cdot M_2^t$ is the quotient of a free module of rank three by the submodule generated by one nonzero column, so $\sV_2\otimes_R\Frac(R)$ has dimension two; an element of $R$ annihilating $\sV_2$ annihilates this vector space and is therefore zero. Hence $\Ann_{\C[s]}\sV_2=(\lambda_2)$.

\emph{Degree three.} Let $R=\C[s]/(\lambda_3)\simeq\C[s_1,s_2,s_3]$. Then $\sV_3=\coker_R(M_3^t)$. The adjugate identity $\operatorname{adj}(M_3^t)M_3^t=\det(M_3)\,I_2$ and \eqref{eq:M3} show that $\Phi$ annihilates $\sV_3$. Conversely, the support of the cokernel of a square matrix over $R$ is exactly the zero locus of its determinant: at a prime not containing $\det M_3$ the matrix is invertible, while at a prime containing it the matrix becomes singular over the residue field and the cokernel is nonzero by Nakayama's lemma. Hence $\sqrt{\Ann_R\sV_3}=\sqrt{(\Phi)}$. The homogeneous quadratic part of $\Phi$ has symmetric matrix
\begin{equation}\label{eq:quadratic-part}
\begin{pmatrix}4&-1&2\\-1&1&-1\\2&-1&1\end{pmatrix},\qquad\text{of determinant }-1 ,
\end{equation}
hence of rank three. A product of two affine linear polynomials has quadratic part of rank at most two, so $\Phi$ is irreducible, and $(\Phi)$ is a prime ideal of $R$. Since $(\Phi)\subseteq\Ann_R\sV_3\subseteq\sqrt{\Ann_R\sV_3}=(\Phi)$, we obtain $\Ann_R\sV_3=(\Phi)$ and
\begin{equation}\label{eq:Ann-V3}
\Ann_{\C[s]}\sV_3=(\lambda_3,\Phi).
\end{equation}

\emph{Degree four.} Here $\sV_4=\C[s]/(\lambda_4,8(s_2+3),-16(2s_1-s_2+s_3+3),-16(s_1+3))$ is cyclic. The three entries of $M_4$ together with $\lambda_4$ give successively $s_2=-3$, $s_1=-3$, $s_3=0$ and $s_4=0$. Therefore
\begin{equation}\label{eq:Ann-V4}
\Ann_{\C[s]}\sV_4=(s_1+3,\,s_2+3,\,s_3,\,s_4).
\end{equation}

\begin{proof}[Proof of Theorem \ref{thm:counter}]
Intersecting the five annihilators computed above gives \eqref{eq:counter-formula}. Each of the five ideals is prime: the first three and the last are generated by linear forms, and $(\lambda_3,\Phi)$ is prime because $\Phi$ is irreducible in $\C[s]/(\lambda_3)\simeq\C[s_1,s_2,s_3]$. They lie on the distinct hyperplanes $\lambda_d=0$, $d=0,1,2,3,4$ (the point $(-3,-3,0,0)$ lies on $\lambda_4=0$), and since $\lambda_d-\lambda_{d'}=d-d'$ is a nonzero constant for $d\neq d'$, they are pairwise comaximal. Hence their intersection equals their product, and it is radical, being an intersection of primes. The subvariety $Z(\lambda_3,\Phi)$ is an irreducible quadric surface in the three-dimensional affine space $(\lambda_3=0)$ which is not a plane, since $\Phi$ is not a product of linear forms; consequently $Z(B^K_f)$, which is the union of the three hyperplanes $(\lambda_d=0)$, $d=0,1,2$, of this quadric surface, and of the point $(-3,-3,0,0)$, is not a finite union of translated linear subvarieties.
\end{proof}

\begin{remark}\label{rem:computer}
Formula \eqref{eq:counter-formula} has also been confirmed by an independent computation: eliminating $x,y,\partial_x,\partial_y$ from the left ideal $\sD_X[s](P_E,P_\delta,x^3,y^3)$ of the Weyl algebra with parameters, using the noncommutative Gr\"obner basis routines of \textsc{Singular} \cite{DGPS}, returns exactly the ideal \eqref{eq:counter-formula}. In the same way one checks for the arrangement \eqref{eq:four-lines} that $B^{-e_1}_f=\big((s_1+1)\lambda_0\lambda_1\big)$, $B^{-(e_1+e_2)}_f=\big((s_1+1)(s_2+1)\lambda_0\lambda_1\lambda_2\big)$, $B^{-2e_1}_f=\big((s_1+1)(s_1+2)\lambda_0\lambda_1\lambda_2\big)$ and $B^{\mon{e_1,e_2}}_f=B^{K_{[4]}}_f=(\lambda_0)$, in accordance with Theorem \ref{thm:main-unit}, Corollary \ref{cor:general-shift}, Theorem \ref{thm:main-coord}.
\end{remark}

\begin{remark}\label{rem:conjectures}
Theorem \ref{thm:counter} shows that neither the algebraic form of Budur's conjecture recalled in \S\ref{sec:intro-monoid} (generation by products of linear polynomials) nor its geometric form, \cite[Conjecture E]{Wu26}, can hold for arbitrary monoid ideals.

The failure of the conjecture is, however, invisible after exponentiation. Let $\Exp\colon\C^4\to(\C^*)^4$, $s\mapsto(e^{2\pi is_1},\dots,e^{2\pi is_4})$. Since $\lambda_d=0$ means $s_1+s_2+s_3+s_4=-2-d\in\Z$, every hyperplane $(\lambda_d=0)$, $d\in\Z$, is mapped onto the subtorus $\{t_1t_2t_3t_4=1\}$; the quadric surface $Z(\lambda_3,\Phi)$ is contained in $(\lambda_3=0)$, and the point $(-3,-3,0,0)$ is mapped to $(1,1,1,1)$. Hence
\[
\Exp\big(Z(B^K_f)\big)=\{t\in(\C^*)^4\mid t_1t_2t_3t_4=1\}
\]
is a single subtorus of codimension one. This is in accordance with \cite[Theorem A]{Wu26}, which identifies $\Exp(Z(B^{K}(\sN_0)))$ with the support of the generalized nearby cycles along $K$ and shows that it is a finite union of translated subtori: the nonlinear component is a feature of the zero locus in $\C^r$ itself, and it cannot be detected by its exponential image. 
\end{remark}

\appendix

\section{The Singular code}\label{sec:appendix}

The \textsc{Singular} \cite{DGPS} code used in Remark \ref{rem:computer}:

{\small
\begin{verbatim}
LIB "nctools.lib";
// The Weyl algebra D_X[s] for X = C^2, with the parameters s1,...,s4.
ring r = 0,(x,y,Dx,Dy,s1,s2,s3,s4),dp;
matrix D[8][8]; D[1,3] = 1; D[2,4] = 1;    // [Dx,x] = 1, [Dy,y] = 1
def A = nc_algebra(1,D); setring A;
option(redSB);

// The arrangement f = (x, y, x+y, x+2y), its Saito basis E, delta,
// and the two logarithmic operators annihilating f^s.
poly f  = x*y*(x+y)*(x+2*y);
poly fx = diff(f,x);  poly fy = diff(f,y);
poly q1 = x^2+6*x*y+6*y^2;   poly q2 = -3*x^2-6*x*y-2*y^2;
poly q3 = x^2+2*x*y-2*y^2;   poly q4 = x^2-2*x*y-2*y^2;
poly PE = x*Dx + y*Dy - (s1+s2+s3+s4);
poly Pd = fy*Dx - fx*Dy - (s1*q1+s2*q2+s3*q3+s4*q4);

// Theorem 6.1: the ideal B_f^K for the monoid ideal K = <3e_1, 3e_2>.
ideal J = eliminate(ideal(PE,Pd,x^3,y^3), x*y*Dx*Dy);

// Comparison with the right-hand side of (6.4).
ring S = 0,(s1,s2,s3,s4),dp;
poly lam = s1+s2+s3+s4+2;
poly Phi = (2*s1-s2+s3+1)*(2*s1-s2+s3+4)+2*(s1+2)*(s2+2);
ideal Bc = intersect(ideal(lam), ideal(lam+1), ideal(lam+2),
                     ideal(lam+3,Phi), ideal(s1+3,s2+3,s3,s4));
Bc = std(Bc);   ideal Jc = std(imap(A,J));
size(reduce(Bc,Jc));   size(reduce(Jc,Bc));  // both 0: the ideals agree

// Theorem 1.1, Corollary 1.2 and Theorem 1.4 for the same arrangement.
setring A;
eliminate(ideal(PE,Pd,x),   x*y*Dx*Dy);   // B_f^{-e_1}
eliminate(ideal(PE,Pd,x*y), x*y*Dx*Dy);   // B_f^{-(e_1+e_2)}
eliminate(ideal(PE,Pd,x^2), x*y*Dx*Dy);   // B_f^{-2e_1}
eliminate(ideal(PE,Pd,x,y), x*y*Dx*Dy);   // B_f^{K_[4]}
\end{verbatim}
}

\end{document}